\documentclass[11pt]{amsart}
\usepackage{a4wide}
\usepackage{amssymb}
\usepackage{amsthm}
\usepackage{amsmath}
\usepackage{amscd}
\usepackage{mathrsfs}
\usepackage{bbm}
\usepackage[all]{xy}
\DeclareMathAlphabet{\mathpzc}{OT1}{pzc}{m}{it}
\usepackage{verbatim}

\numberwithin{equation}{section}

\theoremstyle{plain}
\newtheorem{theorem}{Theorem}[section]
\newtheorem{corollary}[theorem]{Corollary}
\newtheorem{lemma}[theorem]{Lemma}

\theoremstyle{definition}
\newtheorem{definition}[theorem]{Definition}
\newtheorem{remark}[theorem]{Remark}

\theoremstyle{remark}

\newcommand{\R}{\mathbb{R}}
\newcommand{\Q}{\mathbb{Q}}
\newcommand{\Z}{\mathbb{Z}}
\newcommand{\N}{\mathbb{N}}
\newcommand{\C}{\mathbb{C}}

\renewcommand{\H}{\mathbb{H}}

\newcommand{\zxz}[4]{\begin{pmatrix} #1 & #2 \\ #3 & #4 \end{pmatrix}}

\newcommand{\leg}[2]{\left( \frac{#1}{#2} \right)}
\newcommand{\kzxz}[4]{\left(\begin{smallmatrix} #1 & #2 \\ #3 & #4\end{smallmatrix}\right) }

\newcommand{\im}{\operatorname{Im}}
\newcommand{\re}{\operatorname{Re}}

\newcommand{\calG}{\mathcal{G}}
\newcommand{\calH}{\mathcal{H}}

\newcommand{\calQ}{\mathcal{Q}}

\newcommand{\fraka}{\mathfrak a}
\newcommand{\frakb}{\mathfrak b}

\newcommand{\frake}{\mathfrak e}

\newcommand{\bs}{\backslash}

\newcommand{\tr}{\operatorname{tr}}

\newcommand{\SL}{\operatorname{SL}}
\newcommand{\GL}{\operatorname{GL}}
\newcommand{\Sp}{\operatorname{Sp}}

\newcommand{\Sym}{\operatorname{Sym}}

\newcommand{\diag}{\operatorname{diag}}

\newcommand{\sig}{\operatorname{sig}}

\newcommand{\id}{\operatorname{id}}

\begin{document}

\title[The standard zeta function attached to a vector valued modular form ]{On analytic properties of the standard zeta function attached to a vector valued modular form}
\author{Oliver Stein}
\address{Fakult\"at f\"ur Informatik und Mathematik\\ Ostbayerische Technische Hochschule Regensburg\\Galgenbergstrasse 32\\93053 Regensburg\\Germany}
\email{oliver.stein@oth-regensburg.de}
\thanks{On behalf of all authors, the corresponding author states that there is no conflict of interest.}
\subjclass[2020]{11F27 11F25 11F66 11M41}

\begin{abstract}
  We proof a Garrett-B\"ocherer decomposition of a vector valued Siegel Eisenstein series $E_{l,0}^2$ of genus 2 transforming with the Weil representation of $\Sp_2(\Z)$ on the group ring $\C[(L'/L)^2]$. We show that the standard zeta function associated to a vector valued common  eigenform $f$ for the Weil representation can be meromorphically continued to the whole $s$-plane and that it satisfies a functional equation. The proof is based on an integral representation of this zeta function in terms of $f$ and $E_{l,0}^2$. 
\end{abstract}

\maketitle

\section{Introduction}
Vector valued modular forms transforming with the Weil representation   play a prominent role in the theory of Borcherds products, see e. g. \cite{Bor} or \cite{Br1}. Since then, a lot of research  regarding this type of modular forms has been done. Among other things,  a theory of newforms  was developed (\cite{Br2}). In \cite{BS} the foundations of a theory of  Hecke operators was laid. The Fourier coefficients of Eisenstein series were calculated in \cite{BK} and the analytic properties of their non-holomorphic versions were studied in detail in \cite{St}. Also, several relations between these modular forms and  scalar valued modular forms have been established (see e. g. \cite{Sch1}).   One important part of the theory of modular forms that has not yet been addressed (to the best of knowledge) is their relation to $L$-functions. We can find in the literature several ways to associate an $L$-function to a modular form (see e. g.  \cite{Bo1}, \cite{BM}, \cite{PSR}). Such $L$-functions are a central object in number theory and were studied extensively in several aspects among them their analytic properties.  The present paper can be seen as a contribution to these investigations. Its main objective is  to examine the  analytic properties of a certain $L$-function  of a vector valued  Hecke eigenform. 

To describe the main results more closely, we introduce some notation. Let $L$ be non-degenerate even lattice of type $(b^+, b^-)$ and level $N$, equipped with a bilinear form  and associated quadratic form $q$. By $L'$ we denote the dual lattice of $L$. We assume further that the rank and the signature $\sig(L)=b^+-b^-$ of $L$ is even. The Weil representation $\rho_{L,1}$ associated to $L$ is a unitary representation of $\Gamma_1 = \SL_2(\Z)$ on the group ring $\C[L'/L]$,
\[
\rho_{L,1}: \Gamma_1\rightarrow \GL(\C[L'/L]).
\]
For reasons, which will become apparent later,  we introduce a more general  Weil representation $\rho_{L,n}$ of the symplectic group $\Gamma_n = \Sp(n, \Z)$ on the group ring $\C[(L'/L)^n]$ (cf. Definition \ref{def:weil_finite_repr} in Chapter \ref{sec:finite_weil_rep}).

Now let $l\in \Z$ be even. A holomorphic function $f: \H\rightarrow \C$ is called a vector valued modular form of weight $k$ and type $\rho_{L,1}$ for the group $\Gamma_1$ if
\[
f(\gamma\tau) = (c\tau+d)^k\rho_{L,1}(\gamma)f(\tau)
\]
for all $\gamma = \kzxz{a}{b}{c}{d}\in \Gamma$, and $f$ is holomorphic at the cusp $\infty$. The space of all such functions is denoted with $M_l(\rho_{L,1})$. For the subspace of cusp forms we write $S_l(\rho_{L,1})$.  For the details see Chapter \ref{sec:vec_val_mfs_eisenstein_series}.

As already mentioned above, in \cite{BS} a Hecke operator $T(M)$ for $M\in GL_2^+(\Q)$ was introduced by the action of a suitable Hecke algebra. Based on this definition, it is possible to recover the main results of the classical Hecke theory. In particular, it is proved that $S_l(\rho_{L,1})$ possesses a basis of common Hecke eigenforms of all  Hecke operators $T\kzxz{d^2}{0}{0}{1}$ for all $d\in \N$ coprime to the level $N$ of $L$. For such an eigenform we have 
\[
f\mid_{k,L}T\kzxz{d^2}{0}{0}{1} = \lambda_d(f)f
\]
for $(d,N)=1$.  The authors in \cite{BS} proposed to associate to $f$ an an $L$-function of the form 
\begin{equation*}
L^N(s, f) = \sum_{\substack{d\in \N\\(d,N)=1}}\lambda_d(f)d^{-s}.
\end{equation*}
$L$-functions of this type were considered in many places in the literature (cf. for example \cite{Bo1}, \cite{Ar} or \cite{BM}). It was suggested in \cite{BS} to establish the usual analytic properties, that is, the meromorphic continuation to the whole $s$-plane and a functional equation of $L^N(s,f)$ by means of a variant of the doubling method (\cite{Bo1}, \cite{Ga} or \cite{PSR}). This was the starting point of the present paper. Its  main content  is the proof of the meropmorphic continuation and a function equation of a modified version of $L^N(s,f)$.

More precisely, we assume that there is a common Hecke eigenform $f$ for {\it all} Hecke operators $T\kzxz{d^2}{0}{0}{1}$, $d\in \N$. This assumption is justified if a multiplicity one theorem (see e. g.  \cite{Mi}, $\S$4.6) is in place. The question under what conditions such a theorem for vector valued modular forms for the Weil representation holds was investigated in \cite{We}. It turns out that a multiplicity one theorem is valid if  the Weil representation decomposes  into irreducible subrepresentations, each subrepresentation occurring with multiplicity one. In Remark \ref{rem:multiplicity_one} we discuss more detailed under what circumstances $\rho_{L,1}$ allows such a decomposition. Under these assumptions we assign to $f$ the $L$-function
\begin{equation}\label{eq:mod_standard_L_function}
Z(s,f) = \sum_{d\in \N}\lambda_d(f)d^{-s},
\end{equation}
which we call according to \cite{Ar} a {\it standard zeta function}.

To study the analytic properties of \eqref{eq:mod_standard_L_function}, we develop a doubling method along the lines of \cite{Bo1} and \cite{Ga}, tailored to the vector valued setting. Note that there are papers addressing this topic for vector valued Siegel modular forms (\cite{Ko1}, \cite{Ko2}). However, our approach deals with $\rho_{L,1}$ as representation requiring  methods, specifically adapted to this representation, causing several technical difficulties along the way. 
We obtain the following pullback formula:
\begin{theorem}
  Let $d\in \Z$ be a positive integer, $D=\kzxz{d}{0}{0}{d^{-1}}$ and $E_{k,0}^{2*}$ be  as in Definition \ref{def:eisenstein_nonholom}.  Then for all $\tau, \zeta\in \H$
  \begin{align*}
      &  E_{l,0}^{2*}(\kzxz{\tau}{0}{0}{\zeta},s) = \\
&      E_{l,0}^{1*}(\tau,s)\otimes E_{l,0}^{1*}(\zeta,s) + \frac{e(\sig(L)/8)}{|L'/L|^{1/2}}\sum_{d\ge 1}\frac{g_d(L)}{g(L)}d^{-l-2s}\mathscr{P}^+_l(-\tau,\zeta,D,s).
    \end{align*}  
\end{theorem}
Here $E_{l,0}^{2*}$ is the non-holomorphic vector valued Siegel Eisenstein series of type $\rho^*_{L,2}$,  $\mathscr{P}^+_L$ means a certain vector valued Poincar\'e series attached to the discriminant form $(L'/L)^2$ (analogous to the Poincar\'e series in \cite{Bo1}) and $E_{l,0}^{1*}$ is the non-holomorphic vector valued Eisenstein series of type $\rho^*_{L,1}$, where $\rho^*_{L,1}$ and $\rho^*_{L,2}$  denotes the dual representation of $\rho_{L,1}$ and $\rho_{L,2}$, respectively.

Based on this theorem, we then  provide an integral representation of $Z(s,f)$ in terms of the Siegel Eisenstein series $E_{l,0}^{2}$:
\begin{theorem}
  Let $l\in 2\Z$, $l\ge 3$, satisfy $2l+\sig(L)\equiv 0\bmod{4}$.  Let $f\in S_l(\rho_{L,1})$ a common eigenform of Hecke operators $T\kzxz{d^2}{0}{0}{1}$ and $E_{l,0}^{2}$ the Eisenstein series in Definition \ref{def:eisenstein_nonholom} below. If $l+2\re(s)>3$, then, for any $\zeta\in \H$,
  \begin{align*}
    & \sum_{\lambda\in L'/L}\left(\int_{\Gamma_1\bs \H}\langle f(\tau)\otimes \frake_\lambda,E_{l,0}^{2}(\kzxz{\tau}{0}{0}{-\overline{\zeta}},\overline{s})\rangle_{2} \im(\tau)^{l}d\mu(\tau)\right)\frake_\lambda \\
    & = K(l,s)\left(\sum_{d\in \N}\lambda_d(f)d^{-l-2s}\right)f(\zeta).
    \end{align*}
\end{theorem}
Now the analytic properties of $E_{l,0}^2$ with respect to $s$ can be transferred to $Z(s,f)$. Note that  for the Siegel Eisenstein series $E_{l,0}^n$ of genus $n$ and type $\rho_{L,n}$ these properties were established in \cite{St}. The functional equation for $E_{l,0}^n$ is rather complicated. Corollary \ref{cor:functional_eq_E_2} specializes it to the case $n=2$, leading to a quite clean formula.

We finish this section with an overview of the subsequent chapters:
After some notations and preliminaries, we introduce in Chapter \ref{sec:finite_weil_rep} the Weil representation $\rho_{L,n}$ for the Siegel modular group $\Sp(n,\Z)$. We prove several relations between $\rho_{L,2}$ and $\rho_{L,1}$, which are vital for the pullback formula above.
The next chapter deals with vector valued Siegel modular forms transforming with $\rho_{L,n}$, the Eisenstein series $E_{l,0}^n$ and certain vector valued Poincar\'e series, the main ingredients of the above stated formulas. After some general remarks, we switch to the case $n=1$ and give a brief account to Hecke operators and explain how to extend their definition to double cosets $\Gamma_1\kzxz{d}{0}{0}{d^{-1}}\Gamma_1$ if $d$ is not coprime to the level $N$.
The second part of this chapter defines and studies the Eisenstein series $E_{l,0}^n$. In particular, we prove a functional equation for $n=2$ and explain how to assign to $E_{l,0}^n$ a value on the Siegel lower half space, which is a necessary technical detail to deduce the main theorems of the paper.
Finally, two vector valued Poincar\'e series are introduced, one of them analogous to a series introduced in \cite{Bo1}, the other related to the first. The main results here are the reproducing formulas in Theorem \ref{thm:properties_poincare}.
In Chapter 5 we derive the Garrett-B\"ocherer pullback formula in our setting. The proof follows the one of \cite{Ga}. In each relevant step the corresponding calculations for $\rho_{L,2}$ have to be included, causing several technical difficulties.
The last chapter then collects all results established before and presents the desired analytic properties of $Z(s,f)$.

\section{Notation and preliminaries}\label{sec:notation}
We use the symbol $e(x)$, $x\in \C$ as an abbreviation for $e^{2\pi ix}$. As usual, by $\overline{z}$ we mean the conjugate complex number of $z$. In this paper, $n$, $m$ are always natural numbers. For any ring $R$, $M_{m,n}(R)$, $R^n$ and $\Sym_n(R)$ are the set of $m\times n$ matrices, the set of row vectors of size $n$ and the set of symmetric matrices in $M_{n,n}(R)$. We write $1_n$ and $0_n$ for the unit matrix and zero matrix of size n, respectively. Moreover, by $A^t$ and $\tr(A)$ we denote the transposed matrix and the trace  of $A$. Also, for any matrix $S\in \Sym_n(\R)$ we write $S>0$ (resp. $S\ge 0)$ if $S$ is positive definite (resp. semi positive definite).  
By $\Sp(n)$ we denote the symplectic group of genus $n$.  We use the symbol $\Gamma_n$ for $\Sp(n,\Z)$. In particular, we have $\Gamma_1 = \SL_2(\Z)$. For $N\in \N$ 
\begin{align*}
 \Gamma_n(N):= \left\{\gamma\in \Gamma_n\;|\; \gamma\equiv 1_n\bmod{N}\right\}
\end{align*}
is the principal congruence subgroup of $\Gamma_n$.  The subgroups
\begin{equation}\label{def:gamma_n_r}
  \Gamma_{n,r}:=\left\{\kzxz{a}{b}{c}{d}\in \Gamma_n\; :\; a = \kzxz{a_1}{0}{a_3}{a_4}, b=\kzxz{b_1}{b_2}{b_3}{b_4},\; c=\kzxz{c_1}{0}{0}{0},\; d=\kzxz{d_1}{d_2}{0}{d_4}\right\},
\end{equation}
where $0\le r\le n$ and  $*_1\in M_{r,r}(\Z), *_2\in M_{r,n-r}(\Z), *_3\in M_{n-r,r}(\Z),$ and $*_4\in M_{n-r,n-r}(\Z)$ are part of the definition of Siegel Eisenstein series. As special cases we have $\Gamma_{n,n} = \Gamma_n$ and 
\begin{equation}\label{def:gamma_infty}
\Gamma_{n,0} = \left\{\kzxz{a}{b}{c}{d}\in \Gamma_n\; |\; c = 0_n\right\}.
\end{equation}

Further,  by $\Gamma_1\kzxz{0}{d^{-1}}{d}{0}$ we mean the subgroup
\begin{align*}
  &\Gamma_1\kzxz{0}{d^{-1}}{d}{0}=\zxz{0}{d^{-1}}{d}{0}\Gamma_1\zxz{0}{d^{-1}}{d}{0}\cap \Gamma_1, 
\end{align*}
where $d$ is a positive integer. 
Note that $\Gamma_1\kzxz{0}{d^{-1}}{d}{0}$ is isomorphic to
\begin{equation}\label{eq:congruence_subgroup}
  \Gamma_1(d) = \kzxz{d}{0}{0}{d^{-1}}^{-1}\Gamma_1\kzxz{d}{0}{0}{d^{-1}}\cap \Gamma_1
\end{equation}
and as a consequence, $\Gamma_1(\kzxz{0}{d^{-1}}{d}{0})\bs \Gamma_1 \cong \Gamma_1(d)\bs\Gamma_1$. The isomorphism is given by
\begin{equation}\label{eq:isomorphism_congruence_subgroup}
\gamma \mapsto \ell(\gamma)=\kzxz{0}{1}{-1}{0}\gamma^{-1}\kzxz{0}{-1}{1}{0}.  
  \end{equation}
It is immediate that $\ell$ is an involution.

The group $\Gamma_m\times \Gamma_n$ can be embedded into $\Gamma_{n+m}$ by the map 
\begin{equation}\label{eq:embed_sl}
   l_{m,n}: \Gamma_m\times \Gamma_n\rightarrow \Gamma_{n+m}, \quad l_{m,n}\left(\begin{pmatrix}a&b\\c&d\end{pmatrix},\begin{pmatrix}a'&b'\\c'&d'\end{pmatrix}\right)=\begin{pmatrix}a&0&b&0\\ 0&a'&0&b'\\c&0&d&0\\0&c'&0&d'\end{pmatrix}.
\end{equation}

Via $l_{m,n}$ we can embed $\Gamma_m$ and $\Gamma_n$ into $\Gamma_{n+m}$:

\begin{equation}\label{eq:embedding_1}0
  \uparrow: \Gamma_m\rightarrow \Gamma_{n+m},\quad
  \begin{pmatrix}a&b\\c&d\end{pmatrix}\mapsto \zxz{a}{b}{c}{d}^{\uparrow}=l_{m,n}\left(\begin{pmatrix}a&b\\c&d\end{pmatrix},\begin{pmatrix}1&0\\0&1\end{pmatrix}\right),
\end{equation}
\begin{equation}\label{eq:embedding_2}
\downarrow:    \Gamma_n\rightarrow \Gamma_{n+m},\quad  \begin{pmatrix}a&b\\c&d\end{pmatrix}\mapsto \zxz{a}{b}{c}{d}^{\downarrow}=l_{n,m}\left(\begin{pmatrix}1&0\\0&1\end{pmatrix},\begin{pmatrix}a&b\\c&d\end{pmatrix}\right).
\end{equation}

Denote with $\H_n:=\{\tau\in M_{n,n}(\C)\; : \; \tau^t = \tau\; \text{ and } \im(\tau)> 0\}$ the Siegel upper half space of genus $n$. It is well known that the group $\Gamma_n$ acts on $\H_n$ by
\[
\left(\begin{pmatrix}a&b\\c&d\end{pmatrix},\tau\right)\mapsto \begin{pmatrix}a&b\\c&d\end{pmatrix}\tau :=(a\tau+b)\cdot(c\tau+d)^{-1}.
\]
 If $n=1$, $\H_n= \H_1$ specializes to the usual upper half plane $\H$. Moreover, $\H\times\H$ can be embedded into $\H_2$ by 
\[
(\tau_1,\tau_2)\mapsto \begin{pmatrix}\tau_1 & 0\\0&\tau_2\end{pmatrix}.
  \]
  Let $g', g=\kzxz{a}{b}{c}{d}\in \Sp(n,\R)$ and $\tau\in H_n$.
In order to define Siegel modular forms, we also need
an automorphy factor $j_n$:
We set 
\[
j_n(g,\tau):=\det(c\tau+d). 
\]
The automorphy factor satisfies the usual cocycle relation 
\[
j_n(g g',\tau) = j_n(g,g'\tau)j_n(g',\tau).
\]
According to \cite{Fr}, Satz 1.4, we have  the following identity
\[
\im(g\tau) = ^t(c\tau+d)^{-1}\im(\tau)\overline{c\tau+d}^{-1},
\]
which yields
\begin{equation}\label{eq:imag_auto}
\det(\im(g\tau)) = |j_n(g,\tau)|^{-2}\det(\im(\tau)).
\end{equation}

For later purposes we introduce  the function 
  \begin{equation}\label{def:holom_func_abbr}
  w\mapsto \varphi_{l,s}(w)=w^{-l}|w|^{-2s}
  \end{equation}
  and denote with $|A|$ the order of any finite set $A$.

\section{The  Weil representation}\label{sec:finite_weil_rep}
In this section we introduce a finite dimensional representation $\rho_{L,n}$  of the symplectic modular group $\Gamma_n$. It is isomorphic to a subrepresentation of the Weil representation (see e. g. \cite{St}, Chapter 3.3) as defined originally in \cite{Wei}, which explains the name ``Weil representation''. We then specialize to the cases of $n=2$ and $n=1$ and study some relations between $\rho_{L,2}$ and $\rho_{L,1}$, which will be crucial for the Garret-B\"ocherer decomposition of the vector valued Siegel Eisenstein series $E_{l,0}^2$.

The following notation will be used this way throughout the whole paper: Let $L$ be a lattice of rank $m$  equipped with a symmetric $\Z$-valued bilinear form $(\cdot,\cdot)$ such that the associated quadratic form
\[
q(x):=\frac{1}{2}(x,x),\quad x\in L,
\]
takes values in $\Z$. We assume that $m$ is even, $L$ is non-degenerate and denote its type by $(b^+,b^-)$ and its signature $b^+-b^-$ by $\sig(L)$. Note that $\sig(L)$ is also even. We stick with these assumptions on $L$ for the rest of this paper unless we state it otherwise. Further, let 
\[
L':=\{x\in V=L\otimes \Q\; :\; (x,y)\in\Z\quad \text{ for all } \; y\in L\}
\]
be the dual lattice of $L$. 
Since $L\subset L'$, the elementary divisor theorem implies that $L'/L$ is a finite group. The modulo 1 reduction of both, the bilinear form $(\cdot, \cdot)$ and the associated quadratic form, defines a $\Q/\Z$-valued bilinear form $(\cdot,\cdot)$ with corresponding $\Q/\Z$-valued quadratic form on $L'/L$. 
For any two elements $\mu,\nu\in (L'/L)^n$ we define
\begin{equation}\label{eq:moment_matrix}
  \begin{split}
  &  (\mu,\nu)=((\mu_i,\nu_j))_{i,j}\in \Sym_n(\Q),\\
  & Q[\mu] = \frac{1}{2}(\mu,\mu)\in \Sym_n(\Q).
    \end{split}
  \end{equation}
It can be easily verified that
\begin{equation}\label{eq:bilinear_form_n_dim}
  \tr(\mu,\nu)=\sum_{i=1}^n(\mu_i,\nu_i)
\end{equation}
defines a $\Q/\Z$-valued bilinear form on $(L'/L)^n$ with associated quadratic form
\begin{equation}\label{eq:quadratic_form_n_dim}
  \tr(Q[\mu]) = \frac{1}{2}\sum_{i=1}^nq(\mu_i). 
\end{equation}
We call $((L'/L)^n,\tr(\mu,\nu))$ a finite quadratic module or a discriminant form. 
We call the discriminant form $(L'/L, (\cdot,\cdot))$  anisotropic, if $q(\mu) = 0$ holds only for $\mu=0$.

The Weil representation $\rho_{L,n}$ is a representation on the group ring $\C[(L'/L)^n]$. We denote its standard basis  by $\{\frake_{\lambda}\}_{\lambda\in (L'/L)^n}$.
The standard scalar product on $\C[(L'/L)^n]$ is given by
  \begin{equation}\label{eq:scalar_product_group_ring}
    \left\langle \sum_{\lambda\in (L'/L)^n}a_\lambda\frake_\lambda,\sum_{\lambda\in (L'/L)^n}b_\lambda\frake_\lambda \right\rangle_n =\sum_{\lambda\in (L'/L)^n}a_\lambda \overline{b_\lambda}.
  \end{equation}
As $\Gamma_n$ is generated by the matrices
\begin{equation}\label{eq:symplectic_group_gen}
S_n=\zxz{0}{-1_n}{1_n}{0},\quad T_n(b)= \zxz{1_n}{b}{0}{1_n}, 
\end{equation}
where $b\in \Sym_n(\Z)$, it is sufficient to define $\rho_{L,n}$ by the action on these  generators. Note that we will use these symbols in the case $n=1$ without the subscript $n$. Clearly, $T(b)$ is equal to $T^b$ in this case. 

\begin{definition}\label{def:weil_finite_repr}
  The representation $\rho_{L,n}$ of $\Gamma_n$ on $\C[(L'/L)^n],$ defined by
\begin{equation}\label{eq:weil_finite_repr}
\begin{split}
&\rho_{L,n}(T_n(b))\frake_\lambda := e(\tr(bQ[\lambda]))\frake_\lambda,\\
& \rho_{L,n}(S_n))\frake_\lambda:=\frac{e(-\frac{n\sig(L)}{8})}{|L'/L|^{n/2}}\sum_{\mu\in (L'/L)^n}e(-\tr(\mu,\lambda))\frake_\mu,
\end{split}
\end{equation}
is called {\it Weil representation}.
\end{definition}

\begin{remark}\label{rem:weil_sign}
\begin{enumerate}
\item
  The action of $m_n(a) = \zxz{a}{0}{0}{(a^{-1})^t}\in \Gamma_n$ with  $a\in \GL_n(\Z)$ via the Weil representation will be needed later on. It is given by
  \begin{equation}\label{eq:weil_m}
    \rho_{L,n}(m_n(a))\frake_\lambda = \chi_V(\det(a))\frake_{\lambda a^{-1}},
  \end{equation}
  where $\chi_V(\det(a)) = \det(a)^{\sig(L)/2}$,  see \cite{Zh}, Def. 2.2. For further details, e. g. regarding the connection of $\rho_{L,n}$ to the Weil representation in \cite{Wei}, see \cite{Zh} or \cite{St}, p. 10.
\item[ii)]
We denote by $N$ the level of the lattice $L$. It is the smallest positive integer such that $Nq(\lambda)\in \Z$ for all $\lambda\in L'$.  One can prove that the Weil representation $\rho_{L,n}$ is trivial on $\Gamma_n(N)$, the principal congruence subgroup of level $N$, cf. \cite{R}, Theorem 2.4. Therefore, $\rho_{L,n}$ factors over the finite group
  \[
  \Gamma_n/\Gamma_n(N)\cong \Sp(n,\Z/N\Z).
  \]
\item[iii)]
  We denote by  $\rho^*_{L,n}$ the dual representation of $\rho_{L,n}$. Since the Weil representation is unitary, we obtain $\rho^*_{L,n}$ from $\rho_{L,n}$ just by replacing the quadratic  module $((L'/L)^n,\tr(\cdot,\cdot))$ with $((L'/L)^2,-\tr(\cdot,\cdot))$. Therefore, any result involving $\rho_{L,n}$ carries over to $\rho^*_{L,n}$.
  Since $\rho_{L,n}$ is unitary with respect to $\langle \cdot, \cdot\rangle_n$, its dual representation  $\rho^*_{L,n}$ is equal to   the complex conjugate of $\rho_{L,n}$ (interpreted as matrices), that is, 
  \[
  \rho^*_{L,n}(g) = \rho^{-1}_{L,n}(g) = \overline{\rho_{L,n}(g)}
  \]
  for all $g\in \Gamma_n$. 
  In particular,
  \begin{align*}
    \rho^*_{L,n}(g)\frake_\lambda &= \sum_{\mu\in (L'/L)^n}\langle\frake_\lambda,\rho^*_{L,n}(g)\frake_\mu\rangle_{n}\frake_\mu = \sum_{\mu\in (L'/L)^n}\overline{\langle\frake_\lambda,\rho_{L,n}(g)\frake_\mu\rangle}_{n}\frake_\mu. 
  \end{align*}
\item[iv)]
  There is an isomorphism on $\Gamma_n$:
  \begin{equation}\label{eq:conjuation}
\gamma\mapsto    \widetilde{\gamma} = \kzxz{-1_n}{0}{0}{1_n}\gamma\kzxz{-1_n}{0}{0}{1_n}.
   \end{equation}
  For the generators of $\Gamma_n$ we have
  \begin{align*}
  \widetilde{S_n} = S_n^{-1} \text{ and } \widetilde{T_n(b)} = T_n(b)^{-1}
  \end{align*}
Since the map \eqref{eq:conjuation} is an isomorphism,  it  follows
 immediately
\begin{equation}\label{eq:weil_conjugate}
 \rho_{L,n}(\widetilde{\gamma})\frake_\lambda = \rho_{L,n}^{-1}(\gamma)= \rho^*_{L,n}(\gamma) 
\end{equation}
for any $\gamma\in \Gamma_n$, where the last equation is due to part iii) of this remark. 
\end{enumerate}
\end{remark}

We now concentrate on  the special cases $n=1$ and $n=2$. It is well known that $\rho_{L,2}$ and $\rho_{L,1}$ are closely related. To describe this relation more explicitly for some elements of $\Gamma_n$ we first recall some facts see from \cite{BY}, p. 647 and \cite{SV}, Chapter 5,  which relate the group rings $\C[(L'/L)^2]$ and $\C[L'/L]$ and as a consequence $\rho_{L,2}$ and $\rho_{L,1}$: The map
\begin{equation}\label{eq:isomorphism_group_rings}
\C[(L'/L)^2]\rightarrow \C[L'/L]\otimes \C[L'/L],\quad \frake_{(\mu,\nu)}\mapsto \frake_\mu\otimes \frake_\nu
\end{equation}
defines an isomorphism, which induces an isomorphism of the representations $\rho_{L,2}$ and $\rho_{L,1}\otimes \rho_{L,1}$.
It is also a known fact that we can equip $\C[L'/L]\otimes \C[L'/L]$ with an inner product by
\begin{equation}\label{eq:scalar_product_rel}
    \langle \frake_{\lambda_1}\otimes \frake_{\mu_1},\frake_{\lambda_2}\otimes\frake_{\mu_2}\rangle_2 = \langle\frake_{\lambda_1},\frake_{\lambda_2}\rangle_{1} \cdot \langle\frake_{\mu_1},\frake_{\mu_2}\rangle_{1}.
\end{equation}
It is then easily checked that the scalar product \eqref{eq:scalar_product_rel} coincides with \eqref{eq:scalar_product_group_ring} and the map \eqref{eq:isomorphism_group_rings} becomes an isometry.

Finally, note that $\C[L'/L]$ can be embedded into $\C[L'/L]\otimes\C[L'/L]$ and therefore into $\C[(L'/L)^2]$ via the map
\begin{equation}\label{eq:embedding_group_ring}
  \frake_\lambda\mapsto \frake_\lambda\otimes \frake_\lambda.
\end{equation}
The image of $\C[L'/L]$ with respect to this map is isomorphic to $\C[L'/L]$. 

Using the  formulas  in \eqref{eq:weil_finite_repr} and the isomorphism \eqref{eq:isomorphism_group_rings}, we can express $\rho_{L,2}$ on the generators of $\Gamma_2$ explicitly in terms of $\rho_{L,1}$.
\begin{lemma}\label{lem:rel_generators}
  We have the following relation between $\rho_{L,2}$ and $\rho_{L,1}$ on the generators of $\Gamma_2$ and $\Gamma_1$, respectively:
  \begin{equation}\label{eq:rel_generators_rho_L}
  \begin{split}
& \rho_{L,2}(S_2)\frake_{(\lambda_1,\lambda_2)} = \rho_{L,1}(S)\frake_{\lambda_1}\otimes \rho_{L,1}(S)\frake_{\lambda_2}\\
&\rho_{L,2}(T_2\kzxz{a_1}{0}{0}{a_2})\frake_{(\lambda_1,\lambda_2)} = \rho_L(T^{a_1})\frake_{\lambda_1}\otimes \rho_L(T^{a_2})\frake_{\lambda_2},
  \end{split}
  \end{equation}
where  $\kzxz{a_1}{0}{0}{a_2}\in M_2(\Z)$. Clearly, an analogous formula holds for $S_2^{-1}$. 
\end{lemma}
\begin{proof}
These identities can be checked by a straightforward computation involving the formulas \eqref{eq:weil_finite_repr} for $n=2$ and $n=1$ and the isomorphism \eqref{eq:isomorphism_group_rings}.
\end{proof}

The results of the next lemma will be crucial in the proof of a Garret-B\"ocherer decomposition of the Siegel Eisenstein series $E_{l,0}^2$. By identifying $\C[(L'/L)^2]$ with $\C[L'/L]\otimes\C[L'/L]$ they provide formulas that link  $\rho_{L,2}(\gamma^{\uparrow})$ and $\rho_{L,2}(\gamma^{\downarrow})$ with $\rho_{L,1}(\gamma)$. 

\begin{lemma}\label{lem:embedding_weil}
The following formulas for the Weil representation $\rho_{L,2}$ hold:
  \begin{enumerate}
    \item[i)]
\begin{equation}
\rho_{L,2}(S^{\uparrow})\frake_{(\lambda_1,\lambda_2)}=\rho_{L,1}(S)\frake_{\lambda_1}\otimes \frake_{\lambda_2},
\end{equation}
\begin{equation}
\rho_{L,2}(S^{\downarrow})\frake_{(\lambda_1,\lambda_2)} = \frake_{\lambda_1}\otimes \rho_{L,1}(S)\frake_{\lambda_2},
\end{equation}
\begin{equation}\label{eq:embedd_T_1}
\rho_{L,2}(T^{\uparrow})\frake_{(\lambda_1,\lambda_2)} = \rho_{L,1}(T)\frake_{\lambda_1}\otimes \frake_{\lambda_2},
\end{equation}
and
\begin{equation}\label{eq:embedd_T_2}
\rho_{L,2}(T^{\downarrow})\frake_{(\lambda_1,\lambda_2)} = \frake_{\lambda_1}\otimes \rho_{L,1}(T)\frake_{\lambda_2}.
\end{equation}
\item[ii)]
For any $\gamma \in \SL_2(\Z)$ we have
\begin{equation}\label{eq:weil_embedd_1}
  \rho_{L,2}(\gamma^{\uparrow})\frake_{(\lambda_1,\lambda_2)}= \rho_L(\gamma)\frake_{\lambda_1}\otimes \frake_{\lambda_2}
\end{equation}
and
\begin{equation}\label{eq:weil_embedd_2}
  \rho_{L,2}(\gamma^{\downarrow})\frake_{(\lambda_1,\lambda_2)} = \frake_{\lambda_1}\otimes\rho_L(\gamma)\frake_{\lambda_2}.
\end{equation}
  \end{enumerate}
\end{lemma}
\begin{proof}
  \begin{enumerate}
    \item[$i)$]
      For arbitrary $a\in\Sym_2(\Z)$ we use the notation
\begin{equation}\label{def:U_2}
  U_2(a) = \zxz{1_2}{0}{a}{1_2}\in \Gamma_2.
\end{equation}
  One can easily verify the decomposition
\begin{align*}
  &S^{\uparrow}=  U_2(b)T_2(-b)U_2(b) \text{ and }\\
  &S^{\downarrow}=  U_2(c)T_2(-c)U_2(c),
\end{align*}
where $b=\kzxz{1}{0}{0}{0}$ and $c=\kzxz{0}{0}{0}{1}$.  The matrix $U_2(a)$ can be written as product of the generators  \eqref{eq:symplectic_group_gen} 
\begin{align*}
U_2(a) = S_2T_2(-a)S_2^{-1} 
\end{align*}
and therefore
\begin{align*}
  S^{\uparrow} = S_2T_2(-b)S_2^{-1}T_2(-b)S_2T_2(-b)S_2^{-1},\quad S^{\downarrow} = S_2T_2(-c)S_2^{-1}T_2(-c)S_2T_2(-c)S_2^{-1}.
\end{align*}
If we employ the identities \eqref{eq:rel_generators_rho_L} successively for all constituents of the decomposition above, \eqref{eq:weil_finite_repr} applied to $T_2(-b)$ and $T_2(-c)$,  and the bilinearity of the tensor product, we may transfer the evaluation of $\rho_{L,2}(S^{\uparrow})$ to the evaluation of $\rho_{L,1}$ on the corresponding element in $\Gamma_1$:
\begin{equation}
  \begin{split}
    \rho_{L,2}(S^{\uparrow})\frake_{(\lambda_1,\lambda_2)} &= \rho_{L,1}(ST^{-1}S^{-1}T^{-1}ST^{-1}S^{-1})\frake_{\lambda_1}\otimes \rho_{L,1}(SS^{-1}SS^{-1})\frake_{\lambda_2}\\
 & =\rho_{L,1}(ST^{-1}S^{-1}T^{-1}ST^{-1}S^{-1})\frake_{\lambda_1}\otimes\frake_{\lambda_2}  
  \end{split}
\end{equation}
and accordingly
\begin{equation}
  \begin{split}
    \rho_{L,2}(S^{\downarrow})\frake_{(\lambda_1,\lambda_2)} &=\rho_{L,1}(SS^{-1}SS^{-1})\frake_{\lambda_1} \otimes\rho_{L,1}(ST^{-1}S^{-1}T^{-1}ST^{-1}S^{-1})\frake_{\lambda_2} \\
& =\frake_{\lambda_1}\otimes \rho_{L,1}(ST^{-1}S^{-1}T^{-1}ST^{-1}S^{-1})\frake_{\lambda_2}.  
  \end{split}
\end{equation}
We can further simplify the expressions above by means of \cite{McG}, Lemma 4.6. It states for integers $a,d$ satisfying $ad\equiv 1\bmod{N}$ that
\[
\rho_{L,1}(ST^dS^{-1}T^aST^d)\frake_\lambda = \frac{g_{d}(L)}{g(L)}\frake_{d\lambda}, 
  \]
  where $g_d(L) = \sum_{\lambda\in L'/L}e(dq(\lambda))$ and $g(L) = g_1(L)$.
  If we apply this result for $a=d=-1$, we obtain
  \begin{align*}
    \rho_{L,1}(ST^{-1}S^{-1}T^{-1}ST^{-1}S^{-1})\frake_{\lambda} = \frac{e\left(\frac{\sig(L)}{8}\right)}{\sqrt{|L'/L|}}\sum_{\mu\in L'/L}e((\mu,\lambda))\frac{g_{-1}(L)}{g(L)}\frake_{-\mu}.
  \end{align*}
  By Milgram's formula,  
  \[
  g(L) = \sqrt{|L'/L|}e(\sig(L)/8),
  \]
  we have
  \begin{align*}
    e(\sig(L)/8)\frac{g_{-1}(L)}{g(L)} &= e(-\sig(L)/8),
  \end{align*}
  which implies
  \[
  \rho_{L,1}(ST^{-1}S^{-1}T^{-1}ST^{-1}S^{-1})\frake_{\lambda} = \rho_{L,1}(S)\frake_\lambda. 
  \]
  The identities \eqref{eq:embedd_T_1} and \eqref{eq:embedd_T_2} are due to Lemma \ref{lem:rel_generators} and the fact that
  \[
  T^{\uparrow} = T_2(b) \text{ and } T^{\downarrow} = T_2(c),
  \]
  where $b, c$ have same meaning as before in the proof. 
\item[$ii)$]
  This follows easily from the fact that the embeddings \eqref{eq:embedding_1} and \eqref{eq:embedding_2} are group homomorphisms combined with the corresponding formulas for the generators $S$ and $T$ of $\Gamma_1$. 
  \end{enumerate}
\end{proof}

We end this chapter with the evaluation of  the Weil representation $\rho_{L,2}$ on the element $U_2(D)\in \Gamma_2$, where $D$ is equal to $\kzxz{0}{d}{d}{0}$ and $d$ is a positive integer.  This evaluation will also contribute to the proof of a Garret-B\"ocherer decomposition of the Siegel Eisenstein series $E_{l,0}^2$. 
\begin{lemma}\label{lem:weil_gm}
  Let $d\in \N$, $D=\kzxz{0}{d}{d}{0}\in \Sym_2(\Z)$ and  $U_2(D)\in \Gamma_2$ defined as in  \eqref{def:U_2}.

   Then
   \begin{equation}\label{eq:weil_gm}
     \begin{split}
       \rho_{L,2}^{-1}(U_2(D))\frake_{\lambda_1}\otimes\frake_{\lambda_2} &= \frac{1}{|L'/L|}\sum_{\mu,\nu\in L'/L}e((\mu,\lambda_2-\nu))\frake_{d\mu+\lambda_1}\otimes \frake_\nu\\
    \end{split}
     \end{equation}
\end{lemma}
\begin{proof}

Clearly, $U_2(D)^{-1} = U_2(-D)=S_2T_2(D)S_2^{-1}$. 
    A straightforward calculation  using Lemma \ref{lem:rel_generators} and \eqref{eq:weil_finite_repr} yields
    \begin{align*}
      & \rho_{L,2}(U_2(-D))\frake_{\lambda_1}\otimes\frake_{\lambda_2} \\
      &=     \frac{1}{|L'/L|^2}\sum_{\nu_1,\nu_2\in L'/L}\sum_{\mu_2\in L'/L}e((\mu_2,\lambda_2-\nu_2))\sum_{\mu_1\in L'/L}e((\mu_1,d\mu_2+\lambda_1-\nu_1))\frake_{\nu_1}\otimes\frake_{\nu_2}.\\
    \end{align*}
    The sum over $\mu_1$ is equal to $|L'/L|$ if $\nu_1=d\mu_2+\lambda_1$ and 0 otherwise. Hence,
    \begin{align*}
      \rho_{L,2}(U_2(D)^{-1})\frake_{\lambda_1}\otimes\frake_{\lambda_2} = \frac{1}{|L'/L|}\sum_{\nu_2\in L'/L}\sum_{\mu_2\in L'/L}e((\mu_2,\lambda_2-\nu_2))\frake_{d\mu_2+\lambda_1}\otimes\frake_{\nu_2}. 
    \end{align*}
    as claimed.

\end{proof}

\section{Vector valued Siegel modular forms and Eisenstein series}\label{sec:vec_val_mfs_eisenstein_series}
One major  goal of the present paper is to study the analytic properties of the standard zeta function of a certain vector valued modular form of type $\rho_{L,1}$ (a common eigenform of all Hecke operators $T(m^2)$). The main tool in this regard is the Garret-B\"ocherer decomposition of a non-holomorphic vector valued Siegel Eisenstein series $E_{l,0}^2$. 
In this section we provide the necessary background to prove  the Garret-B\"ocherer decomposition and to define and study the before mentioned standard zeta function. 
To this end, we collect some facts about vector valued Siegel modular forms of type $\rho_{L,n}$ and  introduce the non-holomorphic vector valued Siegel Eisenstein series $E_{l,0}^2$ transforming with $\rho_{L,n}$, which plays a key role in determining the analytic properties of the standard zeta function. As a matter of fact, the zeta function inherits these properties from  $E_{l,0}^2$, which is the reason we state and discuss them in some detail in this section.  
We follow \cite{Zh} and \cite{St}.

\subsection{Vector valued Siegel modular forms of type $\rho_{L,n}$}
Let $l\in \Z$. For $f:\H_n\longrightarrow \C[(L'/L)^n]$ and $g\in \Gamma_n$  we define the usual {\it Petersson slash operator} by
\[
(f\mid^{n}_{l,L} g)(\tau):= j_n(g,\tau)^{-l}\rho_{L,n}(g)^{-1}f(g\tau). 
\]
The Petesson slash operator applied to the dual Weil representation $\rho^*_{L,n}$ will be denoted with $\mid^{n*}_{l,L}$. For $n=1$ we drop the superscript and subscript $n$  and simply write $j(g,\tau)$, $\mid_{l,L}$ or $\mid_{l,L}^*$. 

\begin{definition}\label{def:vec_val_siegel_mod_form}
Let $l\in \Z$. A function $f:\H_n\longrightarrow \C[(L'/L)^n]$ is called a Siegel modular form of weight $l$ with respect to $\Gamma_n$ and $\rho_{L,n}$ if 
\begin{enumerate}
\item
$f$ is holomorphic on $\H_n$,
\item
  $f\mid^{n}_{l,L}g = f\;$ for all $g \in \Gamma_n$,
\item
  $f$ has a Fourier expansion of the form
  \begin{equation}\label{eq:fourier_exp_mod_form}
    f(\tau) = \sum_{\lambda\in (L'/L)^n}\sum_{\substack{t\in\Lambda_n\\ t+Q[\lambda]\ge 0}}a(\lambda, t+ Q[\lambda])e(\tr((t+Q[\lambda])\tau))\frake_\lambda,
  \end{equation}
  where $\Lambda_n$ is the set of half-integral $(n\times n)$-matrices. 
\end{enumerate}
We denote the space of these modular forms with $M_{l}(\rho_{L,n})$. Note that the vanishing condition on the Fourier coefficients is automatically fulfilled for $n\ge 2$ by the Koecher principle.
Moreover, if $a(\lambda, t+Q[\lambda]) = 0$ for all $t\in \Lambda_n$ with $t+Q[\lambda] = 0$, then $f$ is called a cusp form. We use the notation $S_l(\rho_{L,n})$ for the subspace of all cusp forms.
\end{definition}
\begin{remark}
  \begin{enumerate}
\item[i)]
  We could have defined the modular forms in Definition \ref{def:vec_val_siegel_mod_form} also for half integral weight $l$. The following argument shows that $M_{l}(\rho_{L,n})=\{0\}$, if the condition
\begin{equation}\label{eq:even_odd}
  2l\equiv \sig(L)\bmod{2}
\end{equation}
is not satisfied.

For $e_n = \diag(-1,1,\dots,1)\in \GL_n(\Z)$  and $f\in M_l(\rho_{L,n})$ by \eqref{eq:weil_m} we have   
\begin{align*}
 f(\tau) = (f\mid^{n}_{l,L}m_n(e_n))(\tau) = (-1)^{-l}(-1)^{\sig(L)/2}f(\tau)
\end{align*}
for all $\tau\in \H_n$. Thus $2l$ and $\sig(L)$ must either be both even or odd if $f$ is a non-trivial form in $M_l(\rho_{L,n})$. Since we have assumed that the rank of $L$ is even, so is the signature $\sig(L)$ and therefore there are only non-trivial integral weight modular forms.
\item[ii)]
  Since $\rho_{L,n}$ is trivial on $\Gamma_n(N)$, it follows directly from Definition \ref{def:vec_val_siegel_mod_form} that the component functions $f_\lambda$ of $f = \sum_{\lambda\in (L'/L)^n}f_\lambda\frake_\lambda$ are Siegel modular forms of weight $l$ for $\Gamma_n(N)$. 
  \end{enumerate}
\end{remark}

\subsubsection{Petersson scalar product and Hecke operators}\label{subsec:petersson_prod_hecke_op}
For $f, g\in M_l(\rho_{L,n})$, one of which is a cusp form, we define the Petersson scalar product of $f$ and $g$ analogously to \cite{Br1}, Section 1.2.2., by
\begin{equation}\label{def:petersson_scalar_product}
(f,g)_n=  \int_{\Gamma_n\bs \H_n}\langle f(\tau),g(\tau)\rangle_n\det(\im(\tau))^ld\mu_n(\tau).
  \end{equation}
Here $d\mu_n(\tau) = \det(y)^{-n-1} dx dy$ with $\tau = x+iy$ is the usual symplectic volume element. If $n=1$, we omit the subscript $n$. The same arguments as in \cite{Br1} show that the integral \eqref{def:petersson_scalar_product} is well defined.

Let $D= \kzxz{d}{0}{0}{d^{-1}}$, $D'= \kzxz{d^2}{0}{0}{1}$ and $R_d=\kzxz{d}{0}{0}{d}$. We now briefly describe the definition of Hecke operators $T(D')$ and $T(D)$ on $M_l(\rho_{L,1})$ (to the best of my knowledge there is no theory of Hecke operators for $M_l(\rho_{L,n})$, although it might be developed in a similar manner). 
We follow \cite{BS}. All details can be found therein.
According to our assumptions on $L$, we only need to  consider the case of even signature.
Recall that $N$ is the level of $L$. We have to distinguish between the cases $(d,N)=1$ and $(d,N)>1$.

Let
\begin{equation}\label{eq:G_N}
\begin{split}
  \calG(N) =& \left\{M\in \GL_2^+(\Q);\quad \exists n\in \Z\text{ with } 
(n,N)=1 \text{ such that } nM\in M_2(\Z)\right.\\ &\left. \text{ and } 
(\det(nM),N)=1\right\}.
\end{split}
\end{equation}  
If $d$ is coprime to $N$, $(D,1)$ and  $(D',d)$ belong to the group
\begin{equation}\label{eq:Q_N}
  \calQ(N) = 
\left\{(M,r)\in \calG(N)\times (\Z/N\Z)^*;\quad \det(M)\equiv r^2\pmod{N}\right\}.
\end{equation}
We identify $\Gamma_1$ with a subgroup of $\calQ(N)$ via the map $\gamma \mapsto (\gamma,1)$. In \cite{BS} an extension of $\rho_{L,1}$ to the group $\calQ(N)$ is introduced. The condition that the determinant of any element of this group is coprime to the level is crucial. 
Based on this  extension we now define Hecke operators in  terms of the Hecke algebra given by the pair of groups $(\calQ(N), \Gamma_1)$. For any $(g,r)\in \calQ(N)$ and $f\in M_l(\rho_{L,1})$ we set
\begin{equation}\label{eq:hecke_op_coprime}
f\mid_{l,L} T(g,r)= \det(g)^{l/2-1}\sum_{M\in \Gamma_1\bs \Gamma_1 g \Gamma_1} f\mid_{l,L} (M,r).
\end{equation}
Instead of $T(D,1)$ and  $T(D',d)$ we write $T(D)$ and $T(D')$, respectively. 
Taking \cite{BS}, (3.5), into account, it can be easily verified that $\rho_{L,1}(D)$ and $\rho_{L,1}(D')$ are connected via 
\begin{equation}\label{eq:rel_rho_D_}
\rho_{L,1}(D') = \rho_{L,1}(R_d)\rho_{L,1}(D)=\frac{g(L)}{g_d(L)}\rho_{L,1}(D).
\end{equation}
This  immediately yields a relation between $T(D)$ and $T(D')$:
\begin{equation}\label{eq:rel_hecke_T_D}
  T(D) = \frac{g(L)}{g_d(L)}T(D').
  \end{equation}
Replacing $\rho_{L,1}$ with its dual representation gives the corresponding relation
\begin{equation}\label{eq:rel_hecke_T_D_dual}
  T(D) = \frac{g_d(L)}{g(L)}T(D').
  \end{equation}
If $(d,N)>1$, $D$ and  $D'$ do not belong to $\calQ(N)$ and the Weil representation cannot be extended to a suitable subgroup of $\GL_2^+(\Q)$ as before. 
However, it can be extended to the double coset $\Gamma_1D'\Gamma_1$ in the following way (\cite{BS}, Section 5):
\begin{equation}\label{eq:extended_weil_D}
  \rho_{L,1}^{-1}(D')\frake_{\lambda} = \frake_{d\lambda}
\end{equation}
and for $\delta = \gamma D'\gamma'\in \Gamma_1D'\Gamma_1$ we put
\begin{equation}\label{eq:extended_weil_double_coset}
  \rho_{L,1}^{-1}(\delta)\frake_\lambda = \rho_{L,1}^{-1}(\gamma')\rho_{L,1}^{-1}(D')\rho_{L,1}^{-1}(\gamma)\frake_\lambda.
  \end{equation}
Note that \eqref{eq:extended_weil_D} is compatible with the definition of $\rho_{L,1}$ in the case of $(d,N)=1$ (see \cite{BS}, Lemma 3.6).
It can be shown (see \cite{BS}, Section 5) that \eqref{eq:extended_weil_double_coset} is independent of the decomposition of $\delta$.
With these definitions in mind, we define the Hecke operator $T(D')$ exactly as in \eqref{eq:hecke_op_coprime}.
In order to define the operator $T(D)$, we set in consistency with \cite{BS}, (3.5)
\begin{equation}\label{def:weil_S_d}
\rho_{L,1}^{-1}(R_d)\frake_\lambda = \frac{g_d(L)}{g(L)}\frake_\lambda \text{ and } \rho_{L,1}(R_d)\frake_\lambda = \frac{g(L)}{g_d(L)}\frake_\lambda.  
\end{equation}
Accordingly we define
\begin{equation}\label{eq:extended_weil_D_prime}
  \rho_{L,1}^{-1}(D)\frake_\lambda = \frac{g_d(L)}{g(L)}\rho_{L,1}^{-1}(D')\frake_\lambda = \frac{g_d(L)}{g(L)}\frake_{d\lambda}. 
  \end{equation}
Since the $\rho_{L,1}^{-1}(D)$ and $\rho_{L,1}^{-1}(D')$ differ just by a constant factor, 
\begin{equation}\label{eq:extended_weil_dobule_coset_prime}
\rho_{L,1}^{-1}(\gamma D\gamma') = \frac{g_d(L)}{g(L)}\rho_{L,1}^{-1}(\gamma D'\gamma')
\end{equation}
together with  \eqref{eq:extended_weil_double_coset} defines an action of the double coset $\Gamma_1D\Gamma_1$.
Again, we find that the relation \eqref{eq:rel_hecke_T_D} holds.

Finally, note that the analogous identity to $(5.5)$ in \cite{BS} also holds for $D$. It can easily be confirmed that
\begin{equation}\label{eq:dual_weil_scalar_prod}
  \langle \frakb,\rho_{L,1}^{-1}(D)\fraka\rangle_1 = \langle \rho_{L,1}^{-1}(D^{-1})\frakb,\fraka\rangle_1
  \end{equation}
for any $\fraka,\frakb\in \C[L'/L]$.

If $d$ is coprime to $N$, one can show that $\frac{g(L)}{g_d(L)}$ is real and thus $T(D)$ is self-adjoint with respect to $(\cdot, \cdot)_1$.
If $(d,N)>1$, it is in  general not true that $\frac{g(L)}{g_d(L)}$ is a real number. In fact, it can be shown that it is an 8-th root of unity (see \cite{Sch}, Section 3, for these assertions)  In this case, in light of \cite{BS}, Theorem 5.6, we may only conclude that
\begin{equation}\label{eq:TD_self_adjoint}
  (f, g\mid_{l,L}T(D))_1 = \kappa_d(f\mid_{l,L}T(D),g)_1
\end{equation}
holds, where
\[
\kappa_d = \overline{\frac{g(L)}{g_d(L)}}\frac{g_d(L)}{g(L)}. 
\]
In both cases $T(D)$ preservers the space $S_l(\rho_{L,1})$. 

\subsection{Vector valued Eisenstein series}
Eisenstein series and Poincar\'e series are important examples of the vector valued Siegel modular forms in Definition \ref{def:vec_val_siegel_mod_form}. In this subsection we introduce  non-holomorphic versions of these series and study the analytic properties of the real analytic Eisenstein series $E_{l,0}^n$ for $\Gamma_n$ transforming with $\rho_{L,n}$. 

\begin{definition}\label{def:eisenstein_nonholom}
  Let  $l\in 2\Z,\; l\ge n+1,$ satisfy the condition $2l+\sig(L)\equiv 0\bmod{4}.$ For $\tau\in \H_n$ and $s\in \C$ we define by 
\begin{equation}\label{eq:eisenstein_nonholom}
\begin{split}
  E_{l,0}^{n}(\tau,s) &=\sum_{\gamma\in \Gamma_{n,0}\bs\Gamma_n}\det(\im(\tau))^s\frake_0\mid^{n}_{l,L}\gamma \\
  &=\det(\im(\tau))^s\sum_{\gamma\in\Gamma_{n,0}\bs\Gamma_n}|j_n(\gamma,\tau)|^{-2s}j_n(\gamma,\tau)^{-l}\rho_{L,n}^{-1}(\gamma)\frake_0.
\end{split}
\end{equation}
a vector valued non-holomorphic Siegel Eisenstein series of weight $k$ transforming with $\rho_{L,n}$. 
\end{definition}

\begin{remark}
  \begin{enumerate}
  \item[i)]
    From the formulas \eqref{eq:weil_finite_repr}, \eqref{eq:weil_m} and the assumption $2l\equiv -\sig(L)\bmod{4}$ follows  that $\det(\im(\tau))^s\frake_0$ is $\Gamma_{n,0}$-invariant with respect to the slash operator. 
  \item[ii)]
  From \cite{St}, Lemma 3.14,  follows  that $E_{l,0}^n(\tau,s)$ converges absolutely for all $\tau\in \H_n$ and all $s\in \C$ with $\re(s) > \frac{n+1-k}{2}$. The usual argument then shows that $E_{l,0}^n$ is with respect to $\mid_{k,L}$ a $\Gamma_n$-invariant real analytic function. In particular, $\tau \mapsto E_{l,0}^n(\tau,0)$ is a vector valued Siegel modular form in $M_{l}(\rho_{L,n})$.  
\item[iii)]
Although $E_{l,0}^n$ is defined on $\H_n$, we can also assign to $E_{l,0}^n(-\tau)$, $\tau\in \H_n$,   a meaningful value:
To this end, let $E_{l,0}^{n*}$ be the Eisenstein series in Definition \ref{def:eisenstein_nonholom} with respect to the dual Weil representation $\rho^*_{L,n}$. Then, since $\gamma\mapsto \widetilde{\gamma}$ is an isomorphism which preservers $\Gamma_{n,0}$, 
  \begin{align*}
    E_{l,0}^{n*}(\tau,s)= \det(\im(\tau))^s\sum_{\gamma\in\Gamma_{n,0}\bs\Gamma_n}|j_n(\widetilde{\gamma},\tau)|^{-2s}j_n(\widetilde{\gamma},\tau)^{-l}\rho_{L,n}^{*-1}(\widetilde{\gamma})\frake_0,
  \end{align*}
  where $\widetilde{\gamma}$ is given by \eqref{eq:conjuation}. Clearly, $j_n(\widetilde{\gamma},\tau) = j_n(\gamma,-\tau)$ and by Remark \ref{rem:weil_sign}, iv),
  \[
  \rho_{L,n}^{*-1}(\widetilde{\gamma}) = \rho_{L,n}^{-1}(\gamma).
  \]
  Therefore, we put
  \begin{equation}\label{eq:eisenstein_series_negative_argument}
    \begin{split}
    &  E_{l,0}^n(-\tau,s) = ((-1)^n)^{s}E_{l,0}^{n*}(\tau,s) \text{ and }\\
    &  E_{l,0}^n(\overline{\tau},s) = ((-1)^n)^{s}E_{l,0}^{n*}(-\overline{\tau},s).
      \end{split}
    \end{equation}
  With the same arguments 
  \begin{equation}\label{eq:eisenstein_series_conjugate_argument}
    \begin{split}
&     E_{l,0}^{n*}(-\tau,s) = ((-1)^n)^{s}E_{l,0}^{n}(\tau,s) \text{ and }\\  
&    E_{l,0}^{n*}(\overline{\tau},s) = ((-1)^n)^{s}E_{l,0}^{n}(-\overline{\tau},s)
    \end{split}
    \end{equation}
  is justified. 
\end{enumerate}
\end{remark}

The analytic properties of $E_{l,0}^{n}$ with respect to $s$ were investigated in \cite{St}. We summarize the results therein and specialize them to the case $n=2$. In particular, the functional equation of $E_{l,0}^{n}$ is for general $n$ quite complicated, yet for $n=2$ it is much simpler and quite clean. To state the functional equation, we need quite a bit of notation, which we describe here for convenience of the reader in full detail. Most of the notation is needed for the transformation behaviour regarding the primes $p$ dividing $|L'/L|$:

By $S_n$ we mean the symmetric group acting on $I=\{1,\dots,n\}$. Associated to  $\sigma\in S_n$ we have the quantities
\begin{align*}
  c_1(\sigma) = |\{i\in I\; |\; \sigma(i)=i\}|,\quad
  c_2(\sigma) = |\{i\in I\; |\; \sigma(i)\not=i\}|.
  \end{align*}
For any partition $I=I_0\cup \dots \cup I_s$ of $I$ into disjoint $\sigma$-stable subsets, we put
\begin{equation}\label{eq:orders}
  \begin{split}
  &n_r = |I_r|, \\
  &n(r) = \frac{n^{(r)}(n^{(r)}+1)}{2},
\end{split}
  \end{equation}
where $n^{(r)} = n_r + n_{r+1} + \cdots + n_{s}$ and $r\ge 0$. Here we understand $n^{(r)}$ to be zero if $r> s$. Also, the terms  
\begin{equation}\label{eq:constants_local_densities}
  \begin{split}
    & t(\sigma, \{I_i\})= \sum_{r=0}^s|\{(i,j)\in I_r\times I_r\; |\; i < j < \sigma(i), \; \sigma(j) < \sigma(i)\}|,\\
  & \tau(\left\{I_i\right\}) = \sum_{r=1}^s |\{(i,j)\in I_r\times (I_0\cup \cdots \cup I_{r-1})\; |\; j < i\}|,\\
    &  c_1^{(r)}(\sigma)=|\{i\in I_r\cup \cdots \cup I_s\; |\; \sigma(i) = i\}|, \; l\ge 0\\
    & \text{ and }\\
    & c_{1,r}(\sigma) = |\{i\in I_r\; |\; \sigma(i)=i\}|,\\
    & c_{2,r}(\sigma) = |\{i\in I_r\; |\; \sigma(i)> i\}|
  \end{split}
\end{equation}
will occur. Since the latter quantities are part of complicated formulas,  we drop the $(\sigma)$-part and write from now on $c_{1,r}$ and $c_{2,r}$. In terms of $c_{1,r}$ and $c_{2,r}$ the following abbreviations will also be used
\begin{align*}
  A_{i}({\bf I_r}) = \sum_{j=0}^{i-1}(\frac{1}{2}c_{1,j}+c_{2,j},  \quad B_i({\bf I_r}) = -mA_i({\bf I_r}) + n(0) - n(i), \; i=1,\dots, k.
\end{align*}
They appear as argument of the local zeta functions $\zeta_p(s)=(1-p^{-s})^{-1}$:
\begin{equation}\label{eq:local_zeta}
  D_{p,j}(s) = \zeta_p(4A_j({\bf I_r})(s-s_0)-2B_j({\bf I_r})),
\end{equation}
with $s_0=\frac{m}{2}-\rho_n$ and $\rho_n=\frac{n+1}{2}$.  
Attached to a positive integer $k\le n$  and  the  set
\[
\{\nu\}_k = \{(\nu_0=-\sum_{i=1}^{k-1}\nu_i-t, \nu_1,\dots, \nu_{k-1})\in \Z\times \N^{k-1}\; |\; t\in \N\}
\]
there are several further quantities:
Apart from
\begin{equation}\label{eq:notation}
  \begin{split}
  & S^{(k)}_r(\nu)= \sum_{i=0}^r\nu_i =  -t - \sum_{i=r+1}^{k-1}\nu_i,\quad  0\le r\le k-1, \\
& o(\nu)_k = \{0\le r\le k-1\; |\; S_r^{(k)}(\nu) \equiv 1\bmod{2}\},\\
    & e(\nu)_k = \{0\le r \le k-1\; |\; S_r^{(k)}(\nu) \equiv 0\bmod{2}\}
\end{split}
  \end{equation}
we will need the subsets
\[
\{\nu_{i_1,\dots,i_d}\}_k, \quad 0\le i_1<i_2< \dots < i_d\le k-1
\]
of $\{\nu\}_k$, which consist of all tuples $\nu$, where the components $\nu_{i_j},\; j=1,\dots,d$ run over all {\it odd} integers and the remaining components over all {\it even} integers.
To emphasize the dependency of $o(\nu)_k$ and $e(\nu)_k$ on the parity of the components of $\nu$ we will later write $o(i_1,\dots,i_d)_k$ and $e(i_1,\dots,i_d)$ respectively if $\nu\in \{\nu_{i_1,\dots,i_d}\}_k$. 
Note that there is a  misprint in \cite{St} regarding these sets. Incorrectly, $n$ is given as  the upper bound for the indices $i_j$. Whenever the indices $i_j$ appear in the original paper, $n$ should be replaced by $k-1$ as  the upper bound for $i_j$.
Based on the notation above, we use
the abbreviation 
\begin{equation}\label{eq:sum_abb}
\sum_{(\sigma, I, k)} = \sum_{\substack{\sigma\in S_n\\\sigma^2=\id}}\sum_{I=I_0\cup \cdots \cup I_s}\sum_{k=0}^{s+1},
\end{equation}
 put
 \begin{equation}\label{eq:kappa}
\begin{split}   
  \kappa_p(\sigma, I, k) &=2^{-c_1(\sigma)}(1-p^{-1})^{c_1(\sigma)+c_2(\sigma)}p^{-c_2(\sigma)}p^{-\tau(\{I_i\})-t(\sigma,\{I_i\})}\\
&  \times\frac{2^{c_1^{(k)}(\sigma)}2^{\sum_{r=0}^{k-1}c_{1,r}}(1-p^{-1})^{c_1^{(k)(\sigma)}}p^{-\sum_{r=k+1}^s n(r)}p^{\sum_{r=0}^{k-1}(c_{1,r}+2c_{2,r})}}{\prod_{r=k}^s(1-p^{-n(r)})}
\end{split}
\end{equation}
and  further introduce $\Delta_p(\sigma,q,k)=1$ for $k=0$ and
\begin{equation}\label{eq:Delta}
  \begin{split}
  \Delta_p(\sigma,q,k) &= \sum_{0\le i_1\le \dots\le i_d\le k-1}(-1)^{\sum_{r\in e(i_1,\dots, i_d)_k}c_{1,r}}\\
  &\times\leg{-1}{p}^{\frac{(m-2)}{2}\sum_{r\in o(i_1,\dots,i_d)_k}c_{1,r}}\leg{2}{p}^{\sum_{r\in e(i_1,\dots, i_d)_k}c_{1,r}} \\
  &\times p^{\sum_{j=1}^d  2A_{i_j}({\bf I_r})s-B_{i_j}({\bf I_r})},
  \end{split}
  \end{equation}
if $ L'/L\cong (\Z/p\Z)^2$ and zero if $L'/L\cong \Z/p\Z$ for $k\ge 1$. Here the sum $\sum_{0\le i_1\le \dots\le i_d\le k-1}$ runs over all subsets of $\{0,\dots, k-1\}$ indicating the positions of the odd components of $\nu$. This includes the empty set, which covers the case that all components of $\nu$ are even.  Note that the parameter $k$ in the symbol $\Delta_p(\sigma,q,k)$ in \cite{St} is missing, which could give the impression that it does not depend on this parameter. This is reason why we add it here.  

The primes $p$ not dividing $|L'/L|$ are collected in the set $P=\{q \text{ prime }\; |\; q\nmid |L'/L|\}$. For each such prime we put
\begin{equation}\label{eq:zeta_functions}
  \begin{split}
  & a_{n,p}(s) = L_p(s+\rho_n-n,\chi_{V,p})\prod_{k=1}^{[\frac{n}{2}]}\zeta_p(2s-n+2k),\\
    & b_{n,p}(s) = L_p(s+\rho_n,\chi_{V,p})\prod_{k=1}^{[\frac{n}{2}]}\zeta_p(2s+n-2k+1),\\
    \end{split}
  \end{equation}
and form the products
\begin{equation}\label{eq:prod_zeta_functions}
  \begin{split}
    & a_n^{P}(s) = \prod_{p\in P}a_{n,p}(s),\quad b_n^{P}(s) = \prod_{p\in P}b_{n,p}(s).
    \end{split}
\end{equation}
Here $L_p(s,\chi_{V,p}))$ means the local $L$-function given by $(1-\chi_{V,p}(p)p^{-s})$. 
The constant $\xi(s,n,k,P)$ then gathers the contribution of the places belonging to the  primes $p\in P$  and the archimedian place:
\begin{equation}
    \begin{split}
      &   \xi(s,n,k,P) = \frac{(-1)^{k/2}2^{(1-s)n}\pi^{\frac{n(n+1)}{2}}}{|L'/L|^{n/2}}\frac{\Gamma_n(s)}{\Gamma_n(\alpha)\Gamma_n(\beta)}\frac{a_{n}^{P}(s)}{b_{n}^{P}(s)},
    \end{split}
  \end{equation}
where $\alpha = \frac{1}{2}(s+\rho_n+l)$ and $\beta=\frac{1}{2}(s+\rho_n-l)$. We are now ready to state the analytic properties of $E_{l,0}^{n}$.

\begin{theorem}\label{thm:functional_eq_E_n}
  Let $E_{l,0}^{n}(\tau,s)$ be the Eisenstein series as in Definition \ref{def:eisenstein_nonholom} with respect to $\rho_{L,n}$. Then $E_{l,0}^{n}(\tau,s)$ has a meromorphic continuation in $s$ to the whole complex plane. If $L'/L$ is anisotropic and $|L'/L|$  odd, it satisfies the functional equation
\begin{equation}\label{eq:func_eq_vecval_eisteinstein}
\begin{split}
    E_{l,0}^{n}(\tau,s-\frac{l}{2})&=
  \xi(2s-\rho_n,n,l,P)\prod_{p\mid |L'/L|}\sum_{(\sigma,I,k)}\kappa_p(\sigma,I,k)\Delta_p(\sigma,q,k)\\&\times \prod_{j=1}^{k}(D_{p,j}(2s-\rho_n)-1) E_{l,0}^{n}(\tau,\rho_n-s-\frac{l}{2}). 
  \end{split}    
\end{equation}
In this formula we interpret any occurring sum or product with lower index bigger than the upper index to be zero or one, respectively. 
  \end{theorem}

As a corollary we phrase now the corresponding theorem for the case $n=2$. To this end, we introduce some quantities, which comprise the contribution of the primes $p$ dividing $|L'/L|$:
We write $S_2 = \{\id, \sigma\}$ and use the symbol $I_0$ in two ways. On the one hand, it means $I_0=\{1,2\}$ if the partition of $I$ consists only of $I_0$. On the other hand, it means $I_0= \{1\}$ if the partition of $I$ is of the form $I=I_0\cup I_1 = \{1\}\cup \{2\}$. Further, by $(L'/L)_p$ we mean the $p$-component of $L'/L$. For $\tau\in S_2$, any ($\tau$-stable) decomposition $I'$ of $I=\{1,2\}$ and $k\in\{0,1,2\}$ we give explicit expressions for
\begin{equation}\label{eq:local_factor_func_eq_implicit}
K_p(\tau,I',k) = \kappa_p(\tau,I',k)\Delta_p(\tau,q,k)\prod_{j=1}^k(D_{p,j}(2s-\frac{3}{2}) - 1):
\end{equation}
\begin{equation}\label{eq:local_factor_func_eq_1}
  \begin{split}
    &  K_p(\id,I_0,0) = \frac{(1-p^{-1})^4}{1-p^{-3}},\\
    & K_p(\id,I_0\cup I_1,0) = \frac{p^{-1}(1-p^{-1})^3}{1-p^{-3}},\\
    & K_p(\sigma,I_0,0) = \frac{p^{-2}(1-p^{-1})^2}{1-p^{-3}}.
  \end{split}
\end{equation}
For $k=1$ and $k=2$ we have to distinguish the cases $(L'/L)_p\cong (\Z/p\Z)^2$ and $(L'/L)_p\cong \Z/p\Z$. In the latter case we have $K_p(\tau,I',k) = 0$ for all listed quantities in \eqref{eq:local_factor_func_eq_2}. Therefore, the following identities hold only for $(L'/L)_p\cong (\Z/p\Z)^2$.
\begin{equation}\label{eq:local_factor_func_eq_2}
  \begin{split}
    &  K_p(\id,I_0,1) = 2(1-p^{-1})^4(\zeta_p(8s+6)-1),\\
    &  K_p(\sigma, I_0,1)= 2(1-p^{-1})^2(\zeta_p(16s-6)-1),\\
    &  K_p(\id,I_0\cup I_1,1) = p(1-p^{-1})^2((-1)\leg{2}{p}+ \leg{-1}{p}^{(m-2)/2})(\zeta_p(8s-6)-1),\\
    & K_p(\id,I_0\cup I_1,2) =\\
    &4p^2(1-p^{-1})^3(1+(-1)\leg{-1}{p}^{(m-2)/2}\leg{2}{p}p^{s+\frac{m}{2}-2})(\zeta_p(4s-4)-1)(\zeta_p(8s-6)-1).
    \end{split}
  \end{equation}

\begin{corollary}\label{cor:functional_eq_E_2}
  Let $E_{l,0}^{2}(\tau,s)$ be the Eisenstein series as in Definition \ref{def:eisenstein_nonholom} with respect to $\rho_{L,2}$. Then $E_{l,0}^{2}(\tau,s)$ has a meromorphic continuation in $s$ to the whole complex plane. If $L'/L$ is anisotropic and $|L'/L|$  odd, it satisfies the functional equation
  \begin{equation}\label{eq:func_eq_E_2}
    E_{l,0}^{2}(\tau,s-\frac{l}{2})=
    \xi(2s-\frac{3}{2},2.l,P)\prod_{p\mid |L'/L|}\left(C_p(\id, I_0) + C_p(\id,I_0\cup I_1)+ C_p(\sigma,I_0)\right)E_{l,0}^{2}(\tau, \frac{3}{2}-s-\frac{l}{2}),
  \end{equation}
  Here
  \[
  \xi(2-\frac{3}{2},2.l,P) = \frac{(-1)^{l/2}4^{1-s}\pi^3}{|L'/L|}\frac{\Gamma_2(s)}{\Gamma_2(\alpha)\Gamma_2(\beta)}\frac{a_2^P(s)}{b_2^P(s)}
    \]
    with the same meaning as above and
    \begin{align*}
      & C_p(\id,I_0)=K_p(\id,I_0,0) + K_p(\id,I_0,1),\\
      & C_p(\id,I_0\cup I_1) = K_p(\id,I_0\cup I_1,0) + K_p(\id,I_0\cup I_1,1) + K_p(\id,I_0\cup I_1,2),\\
      & C_p(\sigma,I_0) = K_p(\sigma,I_0,0) + K_p(\sigma,I_0,1).
    \end{align*}    
\end{corollary}
\begin{proof}
  As indicated before, for the statement of Corollary \ref{cor:functional_eq_E_2} we have to calculate $K_p(\tau,I',k)$ in \eqref{eq:local_factor_func_eq_implicit} for all $\tau\in S_2$, any $\tau$-stable partition $I'$ of $I=\{1,2\}$ and all $k\in \{0,1,2\}$. Since $\Delta_p(\tau,q,0) = 1$ and the product $\prod_{j=1}^0(D_{p,j}(2s-\frac{3}{2})-1)$ is empty and thus equal to 1, we have
  \[
  K_p(\tau,I',0) = \kappa_p(\tau,I',0).
  \]
  By the definition of $\Delta_p(\tau,q,0)$, this identity holds for both possible cases of $(L'/L)_p$.
  The expressions for $K_p(\tau,I',0)$ in \eqref{eq:local_factor_func_eq_1} can be verified by a direct but tedious calculation.
  Again, the definition of $\Delta_p(\tau,q,k)$ for $k=1,2$ yields $K_p(\tau,I',k)=0$ for all $\tau\in S_2$, all partitions $I'$ of $I$ and $(L'/L)_p\cong \Z/p\Z$. 
  It remains to confirm the equations in \eqref{eq:local_factor_func_eq_2}. This is also a straightforward but tedious  calculation.
  For all these computations we have to take the convention in Theorem \ref{thm:functional_eq_E_n} into account. 
  \end{proof}

\begin{remark}\label{rem:functional_equation_dual}
Note that by Remark 3.18 in \cite{St} the functional equation \eqref{eq:func_eq_E_2} remains valid for the Eisenstein series $E_{l,0}^{2*}$. 
\end{remark}

Since the case $n=2$ lies mainly in our interest, we stick with it for the rest of this section and introduce now  a vector valued analogue of the Poincar\'e series in \cite{Bo1}, Section 2.2, and a related Poincar\'e series. To this end, let $\calH = \H\cup \H^-$, where $\H^-$ is the lower complex half-plane.

\begin{definition}\label{def:poincare_series}
  Let $d\in \Z$ be a positive integer, $D=\kzxz{d}{0}{0}{d^{-1}}\in \GL_2(\Q)^+$, $l\in \Z$ and $s\in \C$ with $\re(s) > \frac{3-k}{2}$. 

\begin{enumerate}
    \item[i)]
    $P_{l}:\H\times \H\rightarrow \C[L'/L]\otimes\C[L'/L]$ is defined by
  \begin{equation}\label{eq:poincare_series_1}
    P_{l}(\tau,\zeta,s)=\sum_{\lambda\in L'/L}\sum_{\gamma\in \Gamma_1}\det\left(\im\kzxz{\tau}{0}{0}{\zeta}\right)^s(\tau+\zeta)^{-l}|\tau+\zeta|^{-2s}\frake_\lambda\otimes\frake_\lambda\mid_{l,L} \gamma_{(\tau,1)},
  \end{equation}
  where the subscript $(\tau,1)$ means that $\gamma$  acts on the first component and with respect to the variable $\tau$. We further define a variant of $P_l$, which occurs in the course of the present paper. 
  It  already appeared in slightly less general  scalar valued  form in Poincar\'e's work (see \cite{Le}).
  
  $P_l^+:\calH\times \calH\rightarrow \C[L'/L]\otimes\C[L'/L]$ with
  \begin{equation}\label{eq:dual_poincare_series_1}
    P^+_{l}(\tau,\zeta,s)=\sum_{\lambda\in L'/L}\sum_{\gamma\in \Gamma_1}\det\left(\im\kzxz{\tau}{0}{0}{\zeta}\right)^s(\tau-\zeta)^{-l}|\tau-\zeta|^{-2s}\frake_\lambda\otimes\frake_\lambda\mid_{l,L} \gamma_{(\tau,1)}.
  \end{equation}

\item[ii)]
  Associated to $P_{l}$  and its variant we define
  \begin{equation}\label{eq:poincare_series_2}
    \begin{split}
        \mathscr{P}_l:\H\times\H\rightarrow \C[L'/L]\otimes\C[L'/L],\quad
      \mathscr{P}_l(\tau,\zeta,D,s)= \sum_{M\in \Gamma_1\bs \Gamma_1D\Gamma_1}P_{l}(\tau,\zeta,s)\mid_{l,L}^*M_{(\zeta,2)},
    \end{split}
    \end{equation}
where the subscript $(\zeta,2)$ indicates that $M$  acts on the second component and with respect to the variable $\zeta$ and $\rho^*_{L,1}$.  Corresponding to $P_l^+$ we put 
\begin{equation}\label{eq:poincare_series_plus}
    \begin{split}
        \mathscr{P}^+_l:\calH\times\calH\rightarrow \C[L'/L]\otimes\C[L'/L],\quad
      \mathscr{P}^+_l(\tau,\zeta,D,s)= \sum_{M\in \Gamma_1\bs \Gamma_1D\Gamma_1}P^+_{l}(\tau,\zeta,s)\mid_{l,L}^*M_{(\zeta,2)}.
    \end{split}
\end{equation}
    \end{enumerate}
\end{definition}

\begin{remark}\label{rem:poincare_series}
  \begin{enumerate}
    \item[i)]
The series $P_l$ can be understood as vector valued version of the Poincar\'e series $P_r^k(Z,W,g,s)$ introduced in \cite{Bo1}, Section 2.2. It can also be interpreted as the non-holomorphic variant of the series $h_{\beta,n}$ defined in \cite{RS}, Def. 5.1.
In terms of  the coefficients of $\rho_{L,1}$ the component function $P_{l}(\tau,\zeta,s)_\mu$ can be written as
  \begin{equation}\label{eq:component_poincare_1}
   \det(\im\kzxz{\tau}{0}{0}{\zeta})^s \sum_{\gamma\in \Gamma_1}\langle \rho_{L,1}^{-1}(\gamma)\frake_\lambda,\frake_\mu\rangle j(\gamma,\tau)^{-l}|j(\gamma,\tau)|^{-2s}(\gamma\tau+\zeta)^{-l}|\gamma\tau+\zeta|^{-2s}.
  \end{equation}
  Since $\rho_{L,1}$ factors through the finite group $\SL_2(\Z/N\Z)$, the coefficient $\rho_{\lambda,\mu}^{-1}(\gamma)$ is bounded on $\Gamma_1$. Thus, for matters of convergence of \eqref{eq:component_poincare_1} it suffices to study the
  scalar valued Poincar\'e series
  \[
  \det(\im\kzxz{\tau}{0}{0}{\zeta})^s \sum_{\gamma\in \Gamma_1}j(\gamma,\tau)^{-l}|j(\gamma,\tau)|^{-2s}(\gamma\tau+\zeta)^{-l}|\gamma\tau+\zeta|^{-2s},
  \]
  which is exactly the series $P_1^l(\tau,\zeta,g,s)$ in \cite{Bo1}, Section 2.2, for $g\in \Gamma_1$. It follows from \cite{Bo1}, Sec. 2.2, that each component function of   $P_{l}(\tau,\zeta,s)_\mu$ is absolutely and uniformly convergent for all $s\in \C$ with  $\re(s) > \frac{3-l}{2}$ on any product $V_1(\delta)\times V_1(\delta)$, where
    \[
  V_1(\delta) = \{z = x+iy\in \H\; |\; y\ge \delta \text{ and } x^2\le \frac{1}{\delta}\}
  \]
  with  $\delta>0$, and thereby represents a real analytic function on $\H^2$. Thus, the usual argument shows that $P_l$ transforms under the action of $\Gamma_1$ with respect to the variable $\tau$ like a vector valued modular form of weight $l$ and type $\rho_{L,1}$. 

  Since the sum $\sum_{M\in \Gamma_1(D)\bs \Gamma_1}$ is finite, the same holds for the Poincar\'e series $\mathscr{P}_l$. In terms of the function $\varphi_{l,s}$ (see \eqref{def:holom_func_abbr}), 
  we may write $\mathscr{P}_{l}$ in the following more explicit form
  \begin{equation}
    \begin{split}
      &\sum_{\lambda\in L'/L}\sum_{M\in \Gamma_1(D)\bs \Gamma_1}\\
      &\times\sum_{\gamma\in \Gamma_1}\det(\im\kzxz{\tau}{0}{0}{\zeta})^s\varphi_{l,s}(j(\gamma,\tau)j(M,\zeta)(\gamma\tau+M\zeta))\rho_{L,1}^{-1}(\gamma)\frake_\lambda \otimes \rho_{L,1}^{*-1}(M)\frake_\lambda.
      \end{split}
  \end{equation}
\item[ii)]
  The series $P_l^+$ has similar  properties  as $P_l$. More specifically, it converges absolutely and uniformly on compact subsets of $\calH\times \calH$.
  It is thereby real-analytic on $\calH\times \calH$, leaving out possible poles. Indeed, if $\tau,\zeta$ are both in the upper or lower half-plane, $P_l^+$ has  a pole at $\tau = \zeta$. In all other cases no poles can occur.  Also, the usual argument shows that $P_l^+$ transforms with respect to $\tau\in \H$ like a vector-valued modular form.

  If $\tau$ is an element of $\H$, the before mentioned statements can be deduced from Theorem 3A and Section 3C in Chapter V of \cite{Le}. 

 If $\tau\in\H^-$ and $\zeta\in \H$, $P_l^+$ inherits  the analytic properties of $P_l$. To be specific, by  employing the same argument as in i), we may estimate each of the component functions of $P_l^+$ by 
  \begin{align*}
    \sum_{\gamma\in \Gamma_1}|j(\gamma,\tau)|^{-2s-l}|(\gamma \tau-\zeta)|^{-l-2s} &=  \sum_{\gamma\in \Gamma_1}|j(\gamma,-z)|^{-2s-l}|(\gamma(-z)-\zeta)|^{-l-2s}\\
  &=  \sum_{\gamma\in \Gamma_1}|j(\widetilde{\gamma},z)|^{-2s-l}|(\widetilde{\gamma}z+\zeta)|^{-l-2s},
  \end{align*}
  where we replaced $\tau\in \H^-$ with $-z$, $z \in \H$ and used the fact that $\gamma(-z) = -(\widetilde{\gamma}z)$. The latter sum 
  is up to a constant nothing else but an estimate of the series \eqref{eq:component_poincare_1}.
  If $\tau, \zeta$ are both in $\H^-$, the same reasoning as before shows that $P_l^+$ is bounded by
  \[
  \sum_{\gamma\in \Gamma_1}|j(\widetilde{\gamma},z)|^{-2s-l}|(\widetilde{\gamma}z+\zeta)|^{-l-2s}
  \]
  with $z\in \H$ and $\zeta\in \H^-$. This is up to a constant  an estimate of $P^+_{l|\H\times \H^-}$.
Again, in all considered cases the Poincar\'e series $\mathscr{P}_l^+$ shares the same properties as $P_l^+$. 
    \end{enumerate}
\end{remark}

The following theorem provides some further properties of $P_l^+$ and  $\mathscr{P}_l^+$, which will be vital  later on. 

\begin{theorem}\label{thm:properties_poincare}
  Let $l\in \Z$ and $s\in \C$ with $\re(s)>\frac{3-l}{2}$,  $f\in S_l(\rho_{L,1})$ and
  \begin{equation}\label{eq:constant_reproducing}
    C(l,s) = (-1)^{\frac{l}{2}}2^{2-2s-k}\pi\frac{\Gamma(s+1)}{\Gamma(s+2)}.
      \end{equation}
  \begin{enumerate}
  \item[i)]
    Then we have
    \begin{equation}\label{eq:reproducing_kernel_1}
      \sum_{\lambda\in L'/L}\left(\int_{\Gamma_1\bs \H}\langle f(\tau)\otimes \frake_\lambda,P^+_l(\tau,\overline{\zeta},\overline{s})\rangle_2\im(\tau)^ld\mu(\tau)\right)\frake_\lambda = (-1)^{-s}C(l,s)f(\zeta)
    \end{equation}
    for all $\zeta\in \H$. 
  \item[ii)]
    Let $D = \kzxz{d}{0}{0}{d^{-1}}\in \GL_2(\Q)^+$  and $T(D)$ the Hecke operator defined in Subsection \ref{subsec:petersson_prod_hecke_op}. Then
    \begin{equation}\label{eq:reproducing_kernel_hecke_op}
      \sum_{\lambda\in L'/L}\left(\int_{\Gamma_1\bs \H}\langle f(\tau)\otimes \frake_\lambda,\mathscr{P}^+_l(\tau,\overline{\zeta},D,\overline{s})\rangle_2\im(\tau)^ld\mu(\tau)\right)\frake_\lambda = (-1)^{-s}C(l,s)(f\mid_{l,L}T(D))(\zeta)
    \end{equation}
    for all $\zeta\in \H$. 
  \end{enumerate}
\end{theorem}
\begin{proof}
  $i)$: 
 It can easily be confirmed by direct calculation using the relation \eqref{eq:scalar_product_rel} that
  \begin{equation}\label{eq:scalar_prod_integral_rel}
    \begin{split}
&  \langle f(\tau)\otimes \frake_\lambda,P_l(\tau,\overline{\zeta},s)\rangle_2 =\\
    & \langle f(\tau),\sum_{\mu\in L'/L}\sum_{\gamma\in \Gamma_1}\det\left(\im\kzxz{\tau}{0}{0}{\overline{\zeta}}\right)^{\overline{s}}(\tau-\overline{\zeta})^{-l}|\tau-\overline{\zeta}|^{-2\overline{s}}\frake_\mu\mid_{l,L}\gamma\rangle_1\langle \frake_\lambda\otimes\frake_\mu\rangle_1  
    \end{split}
    \end{equation}
  The last expression shows that the integral in \eqref{eq:reproducing_kernel_1} as Petersson scalar product $(\cdot,\cdot)_1$ is well defined since $f$ is a cusp form and $P^+_l$ transforms like a modular form in $M_l(\rho_{L,1})$ with respect to $\tau$. Using \eqref{eq:scalar_prod_integral_rel} the left-hand side of \eqref{eq:reproducing_kernel_1} becomes
  \begin{equation}\label{eq:integral_expr}
\sum_{\lambda\in L'/L}\left(\int_{\Gamma_1\bs \H}\langle f(\tau),\sum_{\gamma\in \Gamma_1}\det\left(\im\kzxz{\tau}{0}{0}{\overline{\zeta}}\right)^{\overline{s}}(\tau-\overline{\zeta})^{-l}|\tau-\overline{\zeta}|^{-2\overline{s}}\frake_\lambda\mid_{l,L}\gamma\rangle_1\im(\tau)^ld\mu(\tau)\right)\frake_\lambda, 
  \end{equation}
  which can be seen as non-holomorphic version of the kernel operator $K_1$  defined in \cite{RS}, Definition 5.3. To prove the stated assertion, we may use the same unfolding trick as laid out in the proof of Theorem 5.6 in \cite{RS}. In doing so, we obtain
  \begin{equation}\label{eq:reproducing_integral_vec_val}
    \begin{split}
     &     (-1)^{-s}\im(\zeta)^s \int_{\Gamma_1\bs \H}\langle f(\tau),\sum_{\gamma\in \Gamma_1}\varphi_{l,\overline{s}}(j(\gamma,\tau)(\tau-\overline{\zeta}))\rho_{L,1}^{-1}(\gamma)\frake_\lambda\rangle_1\im(\tau)^{l+s}d\mu(\tau) \\
      & = (-1)^{-s}\im(\zeta)^s\int_{\H}\langle f(\tau),(\tau-\overline{\zeta})^{-l}|\tau-\overline{\zeta}|^{-2\overline{s}}\frake_\lambda\rangle_1\im(\tau)^{l+s}d\mu(\tau). 
      \end{split}
  \end{equation}
 Since each component of $f$ is a cusp form for $\Gamma(N)$, the latter integral in \eqref{eq:reproducing_integral_vec_val} can be evaluated explicitly by means of  the proposition in \cite{Bo1}, Section 2.2:
    \begin{equation}\label{eq:reproducing_integral_1}
    (-1)^{-s}\im(\zeta)^s \int_{\H}f_\lambda(\tau)(\overline{\tau}-\zeta)^{-l}|\tau-\overline{\zeta}|^{-2s}\im(\tau)^{l+s}d\mu(\tau) = (-1)^{-s}C(l,s)f_\lambda(\zeta)
    \end{equation}
    for any $\zeta\in \H$. 

    $ii)$:  Similar to $i)$ we have
    \begin{equation}\label{eq:scalar_product_poincare_2}
      \begin{split}
        &        \langle f(\tau)\otimes \frake_\lambda, \mathscr{P}^+_l(\tau,\overline{\zeta}, D, \overline{s})\rangle_{2} =  \sum_{M\in \Gamma_1\bs \Gamma_1D\Gamma_1}\sum_{\mu\in L'/L}\sum_{\gamma\in \Gamma_1} \times\\
     &   \langle f(\tau),\det(\im\kzxz{\tau}{0}{0}{\overline{\zeta}})^{\overline{s}}\varphi_{l,\overline{s}}(j(\gamma,\tau)j(M,\overline{\zeta})(\gamma\tau-M\overline{\zeta}))\rho_{L,1}^{-1}(\gamma)\frake_\mu\rangle_1\langle \frake_\lambda,\rho_{L,1}^{*-1}(M)\frake_\mu\rangle_1.
      \end{split}
    \end{equation}
    Thus, as before,
    \begin{equation}\label{eq:petersson_scalar_product_poincare_2}
      \begin{split}
  &  \int_{\Gamma_1\bs \H}\langle f(\tau)\otimes\frake_\lambda,\mathscr{P}^+_l(\tau,\overline{\zeta},D,\overline{s})\rangle_2\im(\tau)^ld\mu(\tau) \\
      &= (-1)^{-s}\sum_{M\in \Gamma_1(D)\bs\Gamma_1}\im(M\zeta)^sj(M,\zeta)^{-l}\sum_{\mu\in L'/L}\langle\frake_\lambda,\rho_{L,1}^{*-1}(M)\frake_\mu\rangle_1\\
        & \times \int_{\Gamma_1\bs \H}\langle  f(\tau),\sum_{\gamma\in \Gamma_1}\varphi_{l,\overline{s}}(j(\gamma,\tau)(\gamma\tau-M\overline{\zeta}))\rho_{L,1}^{-1}(\gamma)\frake_\mu\rangle_1\im(\tau)^{l+s}d\mu(\tau).
      \end{split}
    \end{equation}
    Unfolding the latter integral and applying \eqref{eq:reproducing_integral_1} subsequently, yields
    \begin{equation}\label{eq:reproducing_hecke_op_id}
      \begin{split}
  &\sum_{\lambda\in L'/L}\left(\int_{\Gamma_1\bs \H}\langle f(\tau)\otimes \frake_\lambda,\mathscr{P}^+_l(\tau,\overline{\zeta},D,\overline{s})\rangle_2\im(\tau)^ld\mu(\tau)\right)\frake_\lambda    \\
      & =(-1)^{-s}C(l,s)\sum_{M\in \Gamma_1\bs\Gamma_1D\Gamma_1}\sum_{\mu\in L'/L}j(M,\zeta)^{-l}f_\mu(M\zeta)\sum_{\lambda\in L'/L}\langle\frake_\lambda,\rho^{*-1}_{L,1}(M)\frake_\mu\rangle_1\frake_\lambda. \\
        \end{split}
    \end{equation}
  The right-hand side of \eqref{eq:reproducing_hecke_op_id}  is nothing else than $(-1)^{-s}C(l,s) (f\mid_{l,L}T(D))(\zeta)$. 

\end{proof}

\section{Garrett-B\"ocherer decomposition of vector-valued Siegel Eisenstein series}
In this section we present a decomposition of the Siegel Eisenstein series $E^2_{l,0}$ in terms of $E^1_{l,0}$ and the Poincar\'e series $\mathscr{P}^+_l$. Such a decomposition  was developed by Garrett (\cite{Ga}) and B\"ocherer (\cite{Bo1}) for scalar valued holomorphic and non-holomorphic Siegel Eisenstein series, respectively. It is based on an explicit system of representatives of $\Gamma_{n+m,0}\bs \Gamma_{n+m}/l_{n+m}(\Gamma_m\times \Gamma_n)$ determined by Garrett. Since we are dealing with the same groups for $n=m=1$, we may use Garrett's results. They become considerably easier in this special case, so we  summarize them  here in the following two theorems. Subsequently, we will use  these theorems to state and prove our vector valued version of the Garrett-B\"ocherer decomposition. 

\begin{theorem}\label{thm:double_coset}
  A complete set of representatives for the double coset $\Gamma_{2,0}\bs \Gamma_2/l_{1,1}(\Gamma_1\times \Gamma_1)$ is given by
  \[
  \bigcup_{r=0}^1\mathscr{M}_r,
  \]
  where
  \begin{equation}\label{eq:double_coset_representative}
    \mathscr{M}_0 = \{U_2(0_2)\}\text{ and } \mathscr{M}_1 = \left\{U_2\kzxz{0}{d}{d}{0}\; |\; d\in \N\right\}.  
    \end{equation}
\end{theorem}

Let $g_d\in \mathscr{M}_i,\; d\in \N_0$. Tailored to the case $n=m=1$ the theorem in \cite{Ga}, $\S$3, specifies a complete set of left coset representatives of $\Gamma_{2,0}\backslash \Gamma_{2,0}\; g_d\; l_{1,1}(\Gamma_1\times \Gamma_1).$

\begin{theorem}\label{thm:left_coset}
  Let $d\in \N_0,\; r\in \{0,1\}$ and $g_d\in \mathscr{M}_r$.  Then a complete set of representatives of the $\Gamma_{2,0}$-left cosets in  $\Gamma_{2,0}\; g_d\;l_{1,1}(\Gamma_1\times \Gamma_1)$ is given by
  \begin{equation}\label{eq:system_left_cosets}
    \left\{g_d\;(\gamma g)^{\uparrow}(M h)^{\downarrow}\;|\; \gamma\in \Gamma_r,\; g\in \Gamma_{1,r}\bs\Gamma_1\;,M\in \Gamma_r\kzxz{0}{d^{-1}}{d}{0}\bs \Gamma_r,\; h\in \Gamma_{1,r}\bs \Gamma_1\right\}.
    \end{equation}
\end{theorem}
Note that we use the convention that $\Gamma_0= \{1_2\}$ and accordingly $\Gamma_0\kzxz{0}{d}{d^{-1}}{0}\bs \Gamma_0 = \{1_2\}$. We also have that $\Gamma_{1,1}\bs \Gamma_1 = \{1_2\}$.

Based on these theorems a vector valued variant of the  pullback formula  in  \cite{Ga}, $\S 5$, and \cite{Bo1}, involving the  Siegel Eisenstein series \eqref{eq:eisenstein_nonholom}, can be given.

\begin{theorem}\label{thm:pullback}
  Let $d\in \Z$ be a positive integer, $D=\kzxz{d}{0}{0}{d^{-1}}$ and $E_{k,0}^{2*}$ be defined as in Definition \ref{def:eisenstein_nonholom}.  Then for all $\tau, \zeta\in \H$
  \begin{equation}\label{eq:pullback}
    \begin{split}
      &  E_{l,0}^{2*}(\kzxz{\tau}{0}{0}{\zeta},s) = \\
&      E_{l,0}^{1*}(\tau,s)\otimes E_{l,0}^{1*}(\zeta,s) + \frac{e(\sig(L)/8)}{|L'/L|^{1/2}}\sum_{d\ge 1}\frac{g_d(L)}{g(L)}d^{-l-2s}\mathscr{P}^+_l(-\tau,\zeta,D,s).
    \end{split}
    \end{equation}  
\end{theorem}
 \begin{remark}\label{rem:convergence_subseries}
  Since $E_{l,0}^{2}$ is absolutely and uniformly convergent on compact subsets of $\H^2$, the same holds for any subseries occurring on the right-hand side of \eqref{eq:pullback}. In particular, $\sum_{d\in \N}\frac{g_d(L)}{g(L)}d^{-l-2s}\mathscr{P}^+_l(-\tau,\zeta,D,s)$ is normally convergent for $\re(s)> \frac{3-l}{2}$. 
  \end{remark}

 \begin{proof}
  The proof is an adaption of the one  given in \cite{Ga}. All steps concerning the factor of automorphy $j_2$ carry over immediately. The parts in $E_{l,0}^{2*}$ coming from the Weil representation $\rho_{L,2}^*$ have to be treated separately.
  
According to Theorem \ref{thm:double_coset} and Theorem \ref{thm:left_coset} we have
\begin{equation}\label{eq:eis_decomp}
  \begin{split}
E_{l,0}^{2*}(\kzxz{\tau}{0}{0}{\zeta},s)
  &= \sum_{g\in \Gamma_{1,0}\bs \Gamma_1}\sum_{h\in \Gamma_{1,0}\bs \Gamma_1}\det\left(\im\kzxz{\tau}{0}{0}{\zeta}\right)^s\frake_{0}\mid^{2*}_{l, L} g^{\uparrow}h^{\downarrow} \\
    &+  \sum_{d\in\N}\sum_{M\in \Gamma_1(d)\bs \Gamma_1}\sum_{\gamma\in \Gamma_1}\det\left(\im\kzxz{\tau}{0}{0}{\zeta}\right)^s\frake_{0}\mid^{2*}_{l,L}U_2\kzxz{0}{d}{d}{0}\gamma^{\uparrow}M^{\downarrow}.
    \end{split}
\end{equation}
Note that we have replaced the left cosets over $\Gamma_1\kzxz{0}{d^{-1}}{d}{0}\bs \Gamma_1$ with $\Gamma_1(d)\bs \Gamma_1$. This amounts to replace $M$ with $\ell(M)$ (see Section \ref{sec:notation} for the notation), which only alters the order of summation but not the whole expression.
We consider both of the summands above separately.

Starting with first summand, note that a by direct computation we find  
 
\begin{equation}\label{eq:calc_cocycle_rel}
j_2(g^\uparrow h^{\downarrow},\kzxz{\tau}{0}{0}{\zeta}) 
= j(g,\tau)j(h,\zeta).
\end{equation}

For the part involving the Weil representation $\rho_{L,2}^*$ we exploit \eqref{eq:weil_embedd_1} and \eqref{eq:weil_embedd_2} and infer that
\begin{equation}\label{eq:calc_weil_repr}
  \begin{split}
  \rho_{L,2}^{*-1}(g^{\uparrow}h^{\downarrow})\frake_{(0,0)} &= \rho_{L,2}^{*-1}(g^{\uparrow})(\frake_0\otimes \rho_{L,1}^{*-1}(h)\frake_0)\\
  &= \rho_{L,1}^{*-1}(g)\frake_0\otimes \rho_L^{*-1}(h)\frake_0.
  \end{split}
  \end{equation}

If we insert the right-hand side of \eqref{eq:calc_cocycle_rel} and \eqref{eq:calc_weil_repr}, we may write in terms of \eqref{def:holom_func_abbr}
\begin{align*}
  &  \sum_{g\in \Gamma_\infty\bs \Gamma_1}\sum_{h\in \Gamma_\infty\bs \Gamma_1}\det\left(\im\kzxz{\tau}{0}{0}{\zeta}\right)^s\frake_0\otimes\frake_0\mid^{2*}_{l,L}g^{\uparrow}h^{\downarrow}\\
  & =\sum_{g\in \Gamma_\infty\bs \Gamma_1}\sum_{h\in \Gamma_\infty\bs \Gamma_1}\det\left(\im\kzxz{\tau}{0}{0}{\zeta}\right)^s\varphi_{l,s}(j(g,\tau)j(h,\zeta))\rho_{L,1}^{*-1}(g)\frake_0\otimes \rho_L^{*-1}(h)\frake_0\\
  &=E_{l,0}^{1*}(\tau,s)\otimes E_{l,0}^{1*}(\zeta,s). 
\end{align*}

For the second summand of \eqref{eq:eis_decomp}, a straightforward calculation, using the cocycle relation of $j_2$ twice, yields
\begin{align*}
&j_2\left((U_2\kzxz{0}{d}{d}{0}\gamma^{\uparrow}M^{\downarrow},\kzxz{\tau}{0}{0}{\zeta}\right) 
  = \left(1-(\gamma\tau)\cdot d^2(M\zeta)\right)j(\gamma,\tau)j(M,\zeta).
\end{align*}

On the other hand, by means of Lemma \ref{lem:embedding_weil} and Lemma \ref{lem:weil_gm} we have
\begin{equation}\label{eq:expr_weil}
  \begin{split}
    \rho_{L,2}^{*-1}(U_2\kzxz{0}{d}{d}{0}\gamma^{\uparrow}M^{\downarrow})\frake_0\otimes\frake_0 &= \frac{1}{|L'/L|}\sum_{\mu,\nu\in L'/L}e(-(\mu,\nu))\rho_{L,2}^{*-1}(M^{\downarrow})\rho_{L,2}^{*-1}(\gamma^{\uparrow})\frake_{d\mu}\otimes\frake_\nu
    \\
&=  \frac{1}{|L'/L|}\sum_{\mu,\nu\in L'/L}e(-(\mu,\nu))\rho_{L,2}^{*-1}(M^{\downarrow})(\rho_{L,1}^{*-1}(\gamma)\frake_{d\mu}\otimes \frake_\nu)\\
&=\frac{1}{|L'/L|}\sum_{\mu,\nu\in L'/L}e(-(\mu,\nu))\rho_{L,1}^{*-1}(\gamma)\frake_{d\mu}\otimes \rho_{L,1}^{*-1}(M)\frake_\nu.
\end{split}
  \end{equation}
The transformation $\gamma\mapsto S^{-1}\widetilde{\gamma}$ (see \eqref{eq:conjuation} for the definition of $\widetilde{\gamma}$) leaves the sum over $\gamma\in \Gamma_1$ and therefore the latter summand of \eqref{eq:eis_decomp} invariant.
The subsequent calculations can be easily verified
\begin{align*}
   j_2\left((U_2\kzxz{0}{d}{d}{0}(S\widetilde{\gamma})^{\uparrow}M^{\downarrow},\kzxz{\tau}{0}{0}{\zeta}\right) & =\left(1-(S^{-1}\widetilde{\gamma}\tau)\cdot d^2(M\zeta)\right)j(S^{-1}\gamma,\tau)j(M,\zeta)\\
  &= (\gamma(-\tau) - d^2M\zeta)j(\gamma,-\tau)j(M,\zeta)\\
  &=d(\gamma(-\tau) - (\kzxz{d}{0}{0}{d^{-1}}M)\zeta)j(\gamma,-\tau) j(\kzxz{d}{0}{0}{d^{-1}}M,\zeta).
\end{align*}
Taking \eqref{eq:expr_weil} and \eqref{eq:weil_finite_repr} into account, the corresponding calculations on the level of the Weil representation  yield
\begin{align*}
  &  \rho_{L,2}^{*-1}(U_2\kzxz{0}{d}{d}{0}\gamma^{\uparrow}M^{\downarrow})\frake_0\otimes\frake_0 \\
  &= \frac{1}{|L'/L|}\sum_{\mu,\nu\in L'/L}e(-(\mu,\nu))\rho_{L,1}^{*-1}(\widetilde{\gamma})\rho_{L,1}^*(S)\frake_{d\mu}\otimes \rho^{*-1}_{L,1}(M)\frake_\nu \\
  & =\frac{e(\sig(L)/8)}{|L'/L|^{3/2}}\sum_{\varrho,\nu\in L'/L}\sum_{\mu\in L'/L}e((\mu,d\varrho-\nu))\rho^{-1}_{L,1}(\gamma)\frake_\varrho\otimes\rho_{L,1}^{*-1}(M)\frake_\nu\\
    &=\frac{e(\sig(L)/8)}{|L'/L|^{1/2}}\sum_{\varrho\in L'/L}\rho^{-1}_{L,1}(\gamma)\frake_\varrho \otimes \rho_{L,1}^{*-1}(M)\frake_{d\varrho}.
\end{align*}
For the second equation we have additionally used \eqref{eq:weil_conjugate}.

With \eqref{eq:extended_weil_D_prime} in mind we may then rewrite the last expression in the form
\begin{equation}\label{eq:eis_decomp_weil}
  \frac{e(\sig(L)/8)}{|L'/L|^{1/2}}\frac{g_d(L)}{g(L)}\sum_{\varrho\in L'/L}\rho^{-1}_{L,1}(\gamma)\frake_\varrho \otimes \rho_{L,1}^{*-1}(\kzxz{d}{0}{0}{d^{-1}}M)\frake_{\varrho}.
\end{equation}

Putting together the rearrangements above, we obtain
\begin{align*}
  & \sum_{d\in\N}\sum_{M\in \Gamma_1(d)\backslash \Gamma_1}\sum_{\gamma\in \Gamma_1}\det\left(\im\kzxz{\tau}{0}{0}{\zeta}\right)^s\frake_0\mid^2_{l,L}U_2\kzxz{0}{d}{d}{0}\gamma^{\uparrow}M^{\downarrow} = \\
  &\sum_{d\in\N}\sum_{R\in \Gamma_1\bs \Gamma_1\kzxz{d}{0}{0}{d^{-1}}\Gamma_1}\sum_{\gamma\in \Gamma_1}\det\left(\im\kzxz{\tau}{0}{0}{\zeta}\right)^sd^{-l-2s}\varphi_{l,s}(j(\gamma,-\tau)j(R,\zeta)(\gamma(-\tau) - R\zeta))\times \\
  & \frac{e(\sig(L)/8)}{|L'/L|^{1/2}}\frac{g_d(L)}{g(L)}\sum_{\varrho\in L'/L}\rho^{-1}_{L,1}(\gamma)\frake_\varrho \otimes \rho_{L,1}^{*-1}(R)\frake_{\varrho}.
\end{align*}

 \end{proof}

\section{Standard zeta function of an eigenform}

In \cite{BS}, p. 251, it was proposed to associate a zeta function
\[
\sum_{\stackrel{d\in \N}{(d,N)=1}}\lambda_d(f)d^{-s}
\]
to a common eigenform $f$ of all Hecke operators $T\kzxz{d^2}{0}{0}{1},\; (d,N)=1$, where
\[
f\mid_{l,L}T\kzxz{d^2}{0}{0}{1} = \lambda_d(f)f. 
\]
In this chapter we study the analytic properties of a slightly different zeta function. Attached to an  Eigenform $f$ of {\it all} Hecke operators $T\kzxz{d^2}{0}{0}{1}$, we define
\begin{equation}\label{def:standard_zeta_f}
  Z(s,f)=\sum_{d\in \N}\lambda_d(f)d^{-s}. 
  \end{equation}
In consistency with \cite{Bo1} the zeta function $Z(s,f)$ can be viewed as the standard  zeta function of  $f$.

\begin{remark}\label{rem:multiplicity_one}
  \begin{enumerate}
  \item[i)]
    As was already stated in \cite{BS}, $Z(s,f)$ converges for $\re(s)$ sufficiently large. We obtain this result as a by-product of the subsequent studies of the analytic properties of $Z(s,f)$.
  \item[ii)]
    Under certain assumptions on the discriminant form one can prove the existence of a common eigenform of all Hecke operators $T\kzxz{d^2}{0}{0}{1},\; d\in \N$. This is possible if a multiplicity one theorem for $S_l(\rho_{L,1})$ holds. In \cite{We}, Theorem 41, conditions for the validity of such a theorem are stated. These conditions depend heavily on the decomposition of the Weil representation $\rho_{L,1}$ into irreducible subrepresentations. The multiplicities of these subrepresentations encode the dimension of the common eigenspace for a set of eigenvalues $\lambda_d$ for all Hecke operators $T\kzxz{d^2}{0}{0}{1},\; (d,N)=1$. If all occurring irreducible subrepresentations have multiplicity one, the same holds for the before mentioned dimension of common eigenspaces.
    The decomposition of $\rho_{L,1}$ into irreducible subrepresentations is well known, see e.g. \cite{No}, \cite{NW}. Among other things, it depends on the structure of the discriminant form $L'/L$. If for example each $p$-group of $L'/L$ consists of a single Jordan block of the form $(\Z/p^\lambda\Z, \frac{rx^2}{p^\lambda})$, $\rho_{L,1}$ decomposes into irreducible subrepresentations of multiplicity one. 
    \end{enumerate}
  \end{remark}

\subsection{Analytic properties of $Z(s,f)$}

In this section we will prove that $Z(s,f)$ can be continued meromorphically to the whole $s$-plane. Also, we are able to establish a functional equation of $Z(s,f)$ under the assumption that the discriminant form is anisotropic. The proof is adaptation of the corresponding result in \cite{Bo1} to the vector valued setting. A first step is 

\begin{theorem}\label{thm:orthogonality}
  Let $f\in  S_l(\rho_{L,1})$ be a cusp form with Fourier expansion
\[
f(\tau) = \sum_{\mu\in L'/L}\sum_{\substack{n\in \Z+q(\mu)\\n>0}}a(\mu,n)e(n\tau).
\]
Then for $\frac{l}{2} + \re(s)>1$ we have
  \begin{equation}\label{eq:orthogonality}
    \sum_{\lambda\in L'/L}\left(\int_{\Gamma_1\bs \H}\langle f(\tau)\otimes \frake_\lambda,E_{l,0}^{1*}(-\tau,\overline{s})\otimes E_{l,0}^{1*}(\overline{\zeta},\overline{s})\rangle_{2} \im(\tau)^{l}d\mu(\tau)\right)\frake_\lambda = 0.
  \end{equation}
  Here $E_{l,0}^{1*}(-\tau,\overline{s})$ and $E_{l,0}^{1*}(\overline{\zeta},\overline{s})$ are defined by
  \eqref{eq:eisenstein_series_conjugate_argument}.
  \end{theorem}
\begin{proof}
Similar to \eqref{eq:scalar_prod_integral_rel} and \eqref{eq:scalar_product_poincare_2}  a straightforward calculation using the relation \eqref{eq:scalar_product_rel} gives
%
   \begin{align*}
     \langle f(\tau)\otimes \frake_\lambda,E_{l,0}^{1*}(-\tau,\overline{s})\otimes E_{l,0}^{1*}(\overline{\zeta},\overline{s})\rangle_{2} &=  (-1)^{2s}\langle f(\tau)\otimes \frake_\lambda,E_{l,0}^1(\tau,\overline{s})\otimes E_{l,0}^1(-\overline{\zeta},\overline{s})\rangle_{2}
     \\&= (-1)^{2s}\langle f(\tau), E_{l,0}^1(\tau,\overline{s})\rangle_{1} \langle \frake_\lambda, E_{l,0}^1(-\overline{\zeta},\overline{s})\rangle_{1}.
     \end{align*}
   Replacing this with the integrand in \eqref{eq:orthogonality}, we find that the integral in \eqref{eq:orthogonality} is equal to 
   \begin{equation}\label{eq:petersson_integral_eisenstein}
    (-1)^{2s} \left(\sum_{\lambda \in L'/L}\langle \frake_\lambda, E_{l,0}^1(-\overline{\zeta},\overline{s})\rangle_{1}\frake_\lambda\right)\int_{\Gamma_1\bs \H}\langle f(\tau), E_{l,0}^1(\tau,\overline{s})\rangle_{1}\im(\tau)^{l}d\mu(\tau).
   \end{equation}
   The integral in \eqref{eq:petersson_integral_eisenstein} is the Petesson scalar product of $f$ and $E_{l,0}^1$.  As such, it is well defined. For its evaluation we adapt the calculations in the proof of \cite{Br1}, Proposition 1.5, to  our situation. The usual unfolding argument (see the proof of Thm. 5.6 in \cite{RS}) yields
   \begin{align*}
     \int_{\Gamma_1\bs \H}\langle f(\tau), E_{l,0}^{1}(\tau,\overline{s})\rangle_{1}\im(\tau)^{l}d\mu(\tau) &= \int_{\Gamma_1\bs \H}\sum_{g\in \Gamma_\infty\bs \Gamma_1}\langle f(g\tau),\frake_0\rangle_1\im(g\tau)^{l+s}d\mu(\tau)\\
     &=\int_{\Gamma_\infty\bs \H}\langle f(\tau),\frake_0\rangle_1\im(\tau)^{l+s-2}dxdy.
   \end{align*}
   Inserting the Fourier expansion of $f_0$ into the last integral, we find
   \begin{align*}
     & \int_0^\infty\int_0^1\sum_{\substack{n\in \Z\\n>0}}a(0,n)e(n\tau)\im(\tau)^{l+s-2}dxdy \\
     &= \int_0^\infty \sum_{\substack{n\in \Z\\n>0}}a(0,n)e^{2\pi n y}y^{l+s-2}dy\left(\int_0^1e^{2\pi i nx}dx\right)\\
     &= 0.
     \end{align*}
\end{proof}

In order to study the analytic properties of the standard zeta function, we express it basically as a Petersson scalar product of $f$ with the restricted Eisenstein series $E_{l,0}^2(\kzxz{\tau}{0}{0}{\zeta},s)$. This approach is weĺl known  and has been applied in several settings, see e. g. \cite{Bo1}, \cite{Ar} or \cite{BM}.

\begin{definition}
  Let $f\in  S_l(\rho_{L,1})$ be a cusp form and $E_{l,0}^{2}$  be the Eisenstein series as in Definition \ref{def:eisenstein_nonholom}. For  $\zeta\in \H$ and $l+2\re(s)>3$ we define the integral 
\begin{equation}\label{eq:scalar_product_eisenstein}
  \sum_{\lambda\in L'/L}\left(\int_{\Gamma_1\bs \H}\langle f(\tau)\otimes \frake_\lambda,E_{l,0}^{2}(\kzxz{\tau}{0}{0}{-\overline{\zeta}},\overline{s})\rangle_{2} \im(\tau)^{l}d\mu(\tau)\right)\frake_\lambda.
\end{equation}
\end{definition}

\begin{remark}\label{rem:rerproducing_integral}
  In view of  Theorem \ref{thm:pullback} we may write the integral in \eqref{eq:scalar_product_eisenstein} in the form
  \begin{equation}\label{eq:reproducing_integral}
    \begin{split}
    & \int_{\Gamma_1\bs \H}\langle f(\tau)\otimes \frake_\lambda,E_{l,0}^{2*}(\kzxz{-\tau}{0}{0}{\overline{\zeta}},\overline{s})\rangle_{2} \im(\tau)^{l}d\mu(\tau) \\
    & =   \int_{\Gamma_1\bs \H}\langle f(\tau)\otimes\frake_\lambda,E_{l,0}^{1*}(-\tau,\overline{s})\otimes E_{l,0}^{1*}(\overline{\zeta},\overline{s})\rangle_2\im(\tau)^ld\mu(\tau)  \\
    &+\frac{e(\sig(L)/8)}{|L'/L|^{1/2}}\sum_{d\in \N}\frac{g_d(L)}{g(L)}d^{-l-2s}\int_{\Gamma_1\bs \H}\langle f(\tau)\otimes\frake_\lambda,\mathscr{P}^+_l(\tau,\overline{\zeta},\kzxz{d}{0}{0}{d^{-1}},\overline{s}\rangle_2\im(\tau)^ld\mu(\tau).
      \end{split}
    \end{equation}
It follows from Remark \ref{rem:convergence_subseries},  \eqref{eq:petersson_scalar_product_poincare_2} and   \eqref{eq:petersson_integral_eisenstein} that the integral in \eqref{eq:scalar_product_eisenstein} is well defined. 
\end{remark}

  The pullback formula \eqref{eq:pullback} combined with Theorem \ref{thm:orthogonality} and Theorem \ref{thm:properties_poincare} give rise to the before mentioned integral formula of the standard zeta function of a common Hecke eigenform $f$. 

\begin{theorem}\label{thm:integral_l_series}
  Let $l\in 2\Z$, $l\ge 3$, satisfy $2l+\sig(L)\equiv 0\bmod{4}$.  Let $f\in S_l(\rho_{L,1})$ and $E_{l,0}^{2}$ the Eisenstein series in Definition \ref{def:eisenstein_nonholom}. If $l+2\re(s)>3$, then, for any $\zeta\in \H$,
  \begin{equation}\label{eq:boecherer_garrett_decomp_int}
    \begin{split}
      & \sum_{\lambda\in L'/L}\left(\int_{\Gamma_1\bs \H}\langle f(\tau)\otimes \frake_\lambda,E_{l,0}^{2}(\kzxz{\tau}{0}{0}{-\overline{\zeta}},\overline{s})\rangle_{2} \im(\tau)^{l}d\mu(\tau)\right)\frake_\lambda \\
      &= K(l,s)\sum_{d\in \N}d^{-l-2s}\left(f\mid_{k,L}T\kzxz{d^2}{0}{0}{1}\right)(\zeta),
    \end{split} 
  \end{equation}
  where $K(l,s)= \frac{e(\sig(L)/8)}{|L'/L|^{1/2}}(-1)^{-s}C(l,s)$ and $C(l,s)$ is specified in \eqref{eq:constant_reproducing}.
  Moreover, if $f$ is common eigenform of all Hecke operators $T\kzxz{d^2}{0}{0}{1}$, the right-hand side of the above identity coincides with
  \begin{equation}\label{eq:boecherer_garrett_decomp_eigen}
    K(l,s)\left(\sum_{d\in \N}\lambda_d(f)d^{-l-2s}\right)f(\zeta).
    \end{equation}
  \end{theorem}

\begin{proof}
  This is routine work. We have just to collect the results we established before and put them together.
Equation \eqref{eq:reproducing_integral} in  Remark \ref{rem:rerproducing_integral}  combined with Theorem \ref{thm:orthogonality} and Theorem \ref{thm:properties_poincare} allows us to replace the left-hand side of the above stated identity with
  \[
 \frac{e(\sig(L)/8)}{|L'/L|^{1/2}}(-1)^{-s}C(l,s)\sum_{d\in \N}\frac{g_d(L)}{g(L)}d^{-l-2s}\left(f\mid_{k,L}T\kzxz{d}{0}{0}{d^{-1}}\right)(\zeta).
  \]
  Employing the relation \eqref{eq:rel_hecke_T_D} afterwards, gives the desired result.  
  \end{proof}

We are now in a position to prove a result concerning the analytic properties of $Z(s,f)$. Corollary \ref{cor:functional_eq_E_2} together with \eqref{eq:boecherer_garrett_decomp_int} gives us the means to transfer the analytic properties of $E_{l,0}^2$ to the standard zeta function. We use the same notation as in Corollary \ref{cor:functional_eq_E_2}.

\begin{theorem}\label{thm:analytic_prop_zeta}
  Let $l\in 2\Z,\, l\ge 3$, satisfy $2l+\sig(L)\equiv 0\bmod{4}$ and $f\in S_l(\rho_{L,1})$ a common eigenform of Hecke operators $T\kzxz{d^2}{0}{0}{1}$. Then the Dirichlet series $Z(s,f)$ can be continued meromorphically to the $s$-plane. If additionally $L'/L$ is anisotropic and $|L'/L|$ odd, then
  \[
  \mathscr{Z}(f,s)=K(l,s)Z(f,2s+l) 
  \]
  satisfies the following functional equation
  \begin{equation}\label{eq:functional_equation_zeta}
    \mathscr{Z}(f,s-\frac{l}{2}) = \xi(2s-\frac{3}{2},2.l,P)\prod_{p\mid |L'/L|}\left(C_p(\id, I_0) + C_p(\id,I_0\cup I_1)+ C_p(\sigma,I_0)\right)\mathscr{Z}(f, \frac{3}{2}-s-\frac{l}{2}).
    \end{equation}
\end{theorem}
\begin{proof}
  First, we note that
  \begin{align*}
    \overline{E_{l,0}^2(\kzxz{\tau}{0}{0}{-\overline{\zeta}},\overline{s})} &= E_{l,0}^{2*}(\kzxz{\overline{\tau}}{0}{0}{-\zeta},s).\\
    &= E_{l,0}^2(\kzxz{-\overline{\tau}}{0}{0}{\zeta},s),
  \end{align*}
  where the last equation is due to \eqref{eq:eisenstein_series_conjugate_argument}.
  Clearly, this identity holds for each component function of $E_{l,0}^{2}$:
  \begin{equation}
    \overline{E_{l,0}^2(\kzxz{\tau}{0}{0}{-\overline{\zeta}},\overline{s})}_\mu = E_{l,0}^{2}(\kzxz{-\overline{\tau}}{0}{0}{\zeta},s)_\mu
  \end{equation}
  for any $\mu\in (L'/L)^2$.
  From this follows the first assertion.

  For the functional equation make use of the fact that \eqref{eq:func_eq_E_2} is valid for each component function of $E_{l,0}^{2}$, which can be immediately read off the proof of Theorem 3.16 of \cite{St} and \cite{BY}, Section 2.2 on pages 641 and 642.
  Taking this into account, we obtain the claimed functional equation.
  \end{proof}


\begin{thebibliography}{MCG}
  \bibitem[AZ]{AZ} A. N. Andrianov and V. G. Zhuravlev, {\it Modular forms and Hecke operators}. Translated form the 1990 Russian original by Neal Koblitz, Transactions of Mathematical Monographs, {\bf 145} (1995)
    \bibitem[Ar]{Ar} T. Arakawa, {\it Jacobi Eisenstein Series and a Basis Problem for Jacobi Forms}, Comment. Math. Univ. St. Paul. {\bf 43}, no 2, 181--216, (1994).
    \bibitem[Bo1]{Bo1} S. B\"ocherer, {\it \"Uber die Funktionalgleichung automorpher $L$-Funktionen zur Siegelschen Modulgruppe}, J. Reine  Angew. Math. {\bf 362}, 146--168, (1985).
      \bibitem[Bor]{Bor}
R. Borcherds, \emph{Automorphic forms with singularities on Grassmannians}, 
Inv. Math. \textbf{132} (1998), 491--562.
      
    \bibitem[BM]{BM} T. Bouganis, J. Marzec, {\it On the analytic properties of the standard $L$-function attached Siegel-Jacobi modular forms}, Dok. Math. {\bf 24}, 2613-2684, (2019).
\bibitem[Br1]{Br1} J. H. Bruinier, {\it Borcherds Products on $O(2,l)$ and Chern Classes of Heegner Divisors}, Lecture notes in mathematics, {\bf 1728}, (2002).
\bibitem[Br2]{Br2} J. H. Bruinier, {\it On the converse theorem for Borcherds Products}, J. Algebra {\bf 397}, 315--342, (2014).
\bibitem[BK]{BK}, J. H. Bruinier, M. Kuss, {\it Eisenstein series attached to lattices and modular forms on orthogonal groups}, Manusripta Math. {\bf 106}, no. 4, 443-459, (2001).
\bibitem[BEF]{BEF} J. H. Bruinier, S. Ehlen, E. Freitag, {\it Lattices with many Borchers products}, Math. Comp. {\bf 85}, no. 300, 1953--1981 (2016).
\bibitem[BS]{BS} J. H. Bruinier and O. Stein, {\it The Weil representation and Hecke operators for vector valued modular forms}, Math. Z. {\bf 264}, 249--270, (2010).
\bibitem[BY]{BY} J. H. Bruinier, T. Yang, {\it Faltings Heights of CM cycles  and derivatives of $L$-functions}, Invent. Math. {\bf 177}, 631--681, (2009).
\bibitem[Fr]{Fr} E. Freitag, {\it Siegelsche Modulfunktionen} Grundlehren der Mathematischen Wissenschaften, Springer-Verlag, (1983).
\bibitem[Ga]{Ga} P. Garret, {\it Pullbacks of Eisenstein series}, Applications in Automorphic Forms of Several Variables, Taniguchi Symposium, Katata 1983, Birkh\"auser, (1984).
\%bibitem[HW]{HW} E. Hewitt, K. A. Ross, {\it Abstract Harmonic Analysis, Vol. I}, Second Edition, Springer-Verlag, (1979).
\bibitem[Ko1]{Ko1} N. Kozima, {\it Garret's pullback fomula for vector-valued Siegel modular forms}, J. Number Theory {\bf 128}, (2008).
\bibitem[Ko2]{Ko2} N. Kozima, {\it Standard $L$-functions attached to vector valued Siegel modular forms}, Kodai Math. J. {\bf 23}, 255--265, (2000).
\bibitem[Le]{Le} J. Lehner, {\it Discontinous groups and automorphic functions},  Mathematical Surveys, Number VIII, American Math. Soc. Providence, (1964).
\bibitem[McG]{McG} W. McGraw, {\it The Rationality of vector valued modular forms associated with the Weil representation}, Math. Ann. {\bf 326}, 105--122, (2003).
\bibitem[Mi]{Mi} T. Miyake, {\it Modular Forms}, Springer-Verlag, New York, (1989).
\bibitem[Mu1]{Mu1} A. Murase, {\it $L$-functions attached to Jacobi forms of degree $n$, Part I. The basic identity}, J. reine angew. Math. {\bf 401}, 122-156 (1989).
\bibitem[Mu2]{Mu2} A. Murase, {\it $L$-functions attached to Jacobi forms of degree $n$, Part II. Functional equation}, Math. Ann. 290, 247--276 (1991).
\bibitem[No]{No} A. Nobs, {\it Die irreduziblen Darstellungen $\GL_2(\Z_p)$, insbesondere $\GL_2(\Z_2)$}, Math. Ann. {\bf 229}, 113--133, (1977).
\bibitem[NW]{NW} A. Nobs, J. Wolfahrt, {\it Die irreduziblen Darstellungen der Gruppen $\SL_2(\Z_p)$, insbesondere $\SL_2(\Z_2)$, II. Teil}, Comment. Math. Helvetici {\bf 39}, 491--526, (1976).
  \bibitem[PSR]{PSR} I. Piatetski-Shapiro, S. Rallis, {\it $L$-functions for Classical Groups}, Lecture Notes in Mathematics, {\bf 1254}, Springer-Verlag, Berlin, (1987).  
  \bibitem[R]{R} B. Runge, {\it Theta Functions and Siegel-Jacobi forms}, Acta Math. {\bf 175}, 165--196, (1995).
\bibitem[RS]{RS} N. Raulf and O. Stein, {\it A trace formula for Hecke operators on vector-valued modular forms}, Glasg. Math. J., {\bf 59}, no 1, 143--165, (2017).
\bibitem[Sch]{Sch} N. Scheithauer, {\it The Weil Rerpresentation of $\SL_2(\Z)$ and Some Applications}, IMNR, {\bf 2009}, no. 8, 1488--1545, (2009).
  \bibitem[Sch1]{Sch1} N. Scheithauer, {\it Some constructions of modular forms for the Weil representation of $\SL_2(\Z)$}, Nagoya Math. J. 
\bibitem[SV]{SV} M. Schwagenscheidt and F. V\"olz, {\it Lifting newforms to vector-valued modular forms for the Weil representation}, Int. J. Number Theory,{\bf 11}, 2199--2219, (2015).
\bibitem[St]{St} O. Stein, {\it Analytic properties of Eisenstein series and standard $L$-functions}, to appear in Nagoya Math. J.
  \bibitem[Wei]{Wei} A. Weil, {\it Sur certains groupes d' operateurs unitaires}, Acta Math. {\bf 111}, 143--211 (1964).
  \bibitem[We]{We} F. Werner, {\it Vector valued Hecke theory}, Ph. D. thesis, Technische Universit\"at Darmstadt, (2014).
    \bibitem[Zh]{Zh} W. Zhang, {\it Modularity of generating functions of special cycles on Shimura varieties}, Ph. D. thesis, Columbia University (2009).
\end{thebibliography}
\end{document}